\documentclass[11pt]{article}
\usepackage{amsmath}
\usepackage{xfrac}
\usepackage{amsfonts}
\usepackage{amsthm}
\usepackage{amssymb}
\usepackage{verbatim}
\usepackage{enumerate}
\usepackage{graphicx}
\usepackage{longtable}
\usepackage[all,cmtip]{xy}
\usepackage{upref}
\usepackage{fixltx2e}
\usepackage{hyperref}
\usepackage{faktor}
\usepackage[hang,small,bf]{caption}
\usepackage{mathtools}
\usepackage{nccmath}
\usepackage{hyperref}

\newtheorem{defn}{Definition}
\newtheorem*{defn*}{Definition}

\newtheorem*{corA}{Theorem A1}

\newtheorem*{teor}{Theorem 1 (\cite{AB,Con1,Merry})}

\newtheorem*{teorA}{Theorem A}
\newtheorem*{teorC}{Theorem C}
\newtheorem*{teorB}{Theorem B}
\newtheorem{pro}[defn]{Proposition}

\newtheorem{lemma}[defn]{Lemma}

\newcommand{\tom}{\tilde{\Omega}}

\newcommand{\U}{\mathcal{U}}
\newcommand{\C}{\mathcal{C}}

\newcommand{\K}{\mathrm{K}}
\newcommand{\B}{\mathcal{B}}
\newcommand{\V}{\mathcal{V}}
\newcommand{\W}{\mathcal{W}}

\newcommand{\Z}{\mathcal{Z}}
\newcommand{\rp}{(0,+\infty)}
\newcommand{\R}{\mathbb{R}}
\newcommand{\Q}{\mathrm{Q}}
\newcommand{\trace}{\mathrm{trace}}
\newcommand{\M}{\mathcal{M}}
\newcommand{\Msec}{\mathrm{Sec}^{\Omega}_k}
\newcommand{\Mric}{\mathrm{Ric}^{\Omega}_k}

\newcommand{\Mo}{M^{\Omega}_k}
\newcommand{\speed}{\dot{\gamma}}

\newcommand{\man}{c(g,\sigma)}

\newcommand{\hess}{\Q_{\gamma}(\eta_k)}
\newcommand{\lochess}{\Q_{0}(F_{\M}^{*} \eta_k)}

\newcommand{\prim}{S_k^{\sigma}}

\newcommand{\locprim}{(S_k^{\sigma}\circ \mathrm{\Phi_{\gamma}^{-1}})}

\newcommand{\flgr}{\Psi_{k}}
\newcommand{\pgr}{\mathrm{X}_k}
\newcommand{\dt}{\frac{D}{dt}}

\newcommand{\ken}{\Sigma_k}
\newcommand{\suk}{\Gamma_{\mathfrak{u}}}
\newcommand{\su}{\mathbf{s}^{\mathfrak{u}}}
\newcommand{\av}{\Delta\eta_k}
\newcommand{\pact}{\Delta \prim}

\numberwithin{equation}{section}

\title{Magnetic curvature and existence of a closed magnetic geodesic on low energy levels}
\author{Valerio Assenza \\ Heidelberg University \\ vassenza@mathi.uni-heidelberg.de}
\date{}
\begin{document}

\maketitle
\begin{abstract}
    To a Riemannian manifold $(M,g)$ endowed with a magnetic form $\sigma$ and its Lorentz operator $\Omega$ we associate an operator $M^{\Omega}$, called the \textit{magnetic curvature operator}. Such an operator encloses the classical Riemannian curvature of the metric $g$ together with terms of perturbation due to the magnetic interaction of $\sigma$. From $M^{\Omega}$ we derive the \textit{magnetic sectional curvature} $\mathrm{Sec}^{\Omega}$ and the \textit{magnetic Ricci curvature} $\mathrm{Ric}^{\Omega}$ which generalize in arbitrary dimension the already known notion of magnetic curvature previously considered by several authors on surfaces. On closed manifolds, under the assumption of $\mathrm{Ric}^{\Omega}$ being positive on an energy level below the Ma\~n\'e critical value, with a Bonnet-Myers argument, we establish the existence of a contractible periodic orbit. In particular, when $\sigma$ is nowhere vanishing, this implies the existence of a contractible periodic orbit on every energy level close to zero. Finally, on closed oriented even dimensional manifolds, we discuss about the topological restrictions which appear when one requires $\mathrm{Sec}^{\Omega}$ to be positive.
  \end{abstract}
\section{Introduction and results}
\subsection{Magnetic systems and closed magnetic geodesics}
A magnetic system is the data of $(M,g,\sigma)$, where $(M,g)$ is a closed Riemannian manifold and $\sigma$ is a closed 2-form on $M$ and in this context it is referred as the \textit{magnetic form}. The condition that $\sigma$ is closed generalizes the fact that, in Euclidean space, magnetic fields are divergence-free according to Maxwell's equations. A bundle operator $\Omega : TM \to TM$, called Lorentz operator, is associated with $\sigma$ and the metric $g=\langle \ , \ \rangle $ in the following sense:
\begin{equation} \label{compa}
    \langle v , \Omega(w) \rangle = \sigma(v,w),  \ \forall v,w\in TM. 
\end{equation}
Observe that $\Omega$ is antisymmetric with respect to $g$. In this setting, denoting by $\frac{D}{dt}$ the covariant derivative associated to the Levi-Civita connection $\nabla$, the dynamics of a charged particle, moving on $(M,g)$ under the influence of the magnetic field $\sigma$, is described by the second-order differential equation: 
\begin{equation}\label{mg}
    \frac{D}{dt}\speed=\Omega(\speed)  . 
\end{equation}
A solution $\gamma: \R \to M$ to~\eqref{mg} is called a  \textit{magnetic geodesic}. Observe that when there is no magnetic interaction, i.e. $\Omega=0$, the above equation is reduced to the equation of standard geodesics. The energy $E(p,v)=\frac{1}{2}\langle v , v \rangle $ is constant along magnetic geodesics so that the \textit{magnetic flow} 
\begin{equation}\label{flussomagnetico}
\mathbf{\Phi}^{\sigma}:TM \times \R \to TM,
\end{equation}
obtained by lifting equation \eqref{mg} to $TM$, leaves invariant the sets $\ken= \{ (p,v), \ |v|=\sqrt{2k} \}$, with $k\in (0,+\infty)$. A \textit{closed magnetic geodesic} is a periodic magnetic geodesic and one of the most relevant interests in the theory is the study of existence, multiplicity and free homotopy class of closed magnetic geodesic in a given $\ken$; see for instance \cite{intro2},\cite{intro1} or also \cite{intro3}. In a variational setting, closed magnetic geodesic with energy $k$ are precisely zeros for a suitable closed 1-form $\eta_k$, called the \textit{magnetic action form}, which is defined on $\M$, the space of absolutely continuous loops with $L^2$-integrable derivative and free period \cite{AB,Merry}. Such a form always admits a local primitive and generalizes the differential of the classical action functional which is globally defined for example if the magnetic form $\sigma$ is exact on the base manifold $M$. When the magnetic form is weakly exact, namely $\tilde{\sigma}=p^* \sigma$ is exact on the universal cover $p: \Tilde{M} \to M$, we can extend a local primitive $\prim$ of $\eta_k$ to the whole connected component of contractible elements of $\M$ which we denote by $\M_0$. Such a primitive it is given for instance by 
\begin{equation}\label{primcon}
\prim:\M_0 \to \R  , \  \prim(\gamma)=\int_0^T \Big\{  \frac{|\speed|^2}{2}+k  \Big\} \ \mathrm{d}t + \int_{B^2} D_\gamma^* \sigma  ,
\end{equation} 
where $D_{\gamma} : B^2 \to M$ is an arbitrary capping disk for $\gamma$. Since $\sigma$ is weakly exact if and only if $[\sigma] \in H^2(M,\R)$ vanishes over $\pi_2(M)$, definition \eqref{primcon} is independent of the choice of $D_{\gamma}$. Among the values of the energy, the Ma\~n\'e critical value $c=c(g,\sigma) \in [0,+\infty]$ plays a crucial role in the theory. If $\sigma$ is weakly exact, define $c$ by 
\begin{equation}\label{mane}
c= \inf \{ k\geq0 \ | \ \prim(\gamma)\geq 0, \ \forall \gamma \in \M_0 \}.
\end{equation}
If $\sigma$ is not weakly exact we set $\man=+\infty$. Observe that $c=0$ if and only if $\sigma =0$. Moreover, $c$ is finite if and only if $\tilde{\sigma}$ admits a bounded primitive, i.e. there exists $\tilde{\theta} \in \Omega^1(\tilde{M})$ such that $ \mathrm{d}\tilde{\theta} = \tilde{\sigma}$ and 
\begin{equation}\label{boundedprimitive}
\sup_{x \in \tilde{M}} |\tilde{\theta}_x| < +\infty ,
\end{equation}
where with abuse of notation, we write $|\cdot |$ the dual norm on $\tilde{M}$  induced by the lifted metric $p^*g$. Generally the magnetic dynamics can drastically differ on $\Sigma_k$ depending on whether $k$ is above, below or equal to $c$. For instance, this becomes evident when we look at the existence of closed magnetic geodesics. 

\begin{teor}\label{teorema1} Let $(M,g,\sigma)$ be a magnetic system and $c \in [0,+\infty]$ its M\~an\'e critical value.

\begin{itemize}
\item[(i)] If $\pi_1(M)$ is non trivial and $k>c$, then there exists a closed magnetic geodesic with energy $k$ in each non trivial homotopy class.
\item[(ii)] If $M$ is simply connected and $k>c$, then there exists a contractible closed magnetic geodesic with energy $k$.
\item[(iii)] For almost all $k \in (0,c)$ there exists a contractible closed magnetic geodesic with energy $k$.
\end{itemize}

\end{teor}
In fact, when $c$ is finite the magnetic form admits a primitive on each connected component of $\M$. If $k>c$ such a primitive is bounded from below and satisfies the Palais-Smale condition; thus the existence of a zero for $\eta_k$, which is a minimizer of any primitive, is obtained via classical variational methods through minimization. Below $c$ such compactness conditions are not satisfied anymore and the variational approach becomes extremely delicate. Despite that, for almost all energy $k$, with a Struwe monotonicity argument \cite{struwe}, one can construct converging Palais-Smale sequence which origin from a minimax geometry of $\eta_k$ on the set of loops with short length. In general, it is still an open question if we can extend the existence of closed magnetic geodesic, contractible or not, to the whole interval $(0,c)$. Under additional assumptions, some progress had been made. For exact magnetic systems on compact oriented surfaces, Taimanov in \cite{taim2}, Contreras, Macarini and Paternain in \cite{CMP} and later Asselle and Mazzucchelli in \cite{AsMaz}, show the existence of a closed magnetic geodesic with energy $k$ for all $k\in \rp$. Another relevant result is in the work of Ginzburg-G\"urel \cite{GB}, refined by Usher in \cite{ush}, where they establish for symplectic magnetic forms the existence of contractible closed magnetic geodesics with energy $k$, for all $k$ close to zero. In this work we extend the result of Ginzburg-G\"urel and Usher from symplectic magnetic forms to nowhere vanishing magnetic forms. 
\begin{corA}
    Let $(M,g,\sigma)$ be a magnetic system. If $\sigma$ is nowhere vanishing, then there exists a positive real number $\rho$ such that for every $k \in (0, \rho)$ there exists a contractible closed magnetic geodesic with energy $k$.
\end{corA}
Observe that the assumption on a closed 2-form to be nowhere vanishing presents weak topological obstructions and the set of such systems is very rich. Indeed, in dimension two, being nowhere vanishing is the same as being symplectic, and it is well known that every oriented compact surfaces admits a symplectic form. The situation is drastically different in dimension bigger than 2. There are topological obstructions for a manifold to carry a symplectic form; for instance, it needs to be even-dimensional and all cohomology groups in even degree must be non-zero. On the other hand, as treated in \cite[Section~20.3]{eliashberg}, every smooth manifold of dimension bigger or equal than 3 admits a nowhere vanishing closed 2-form in each cohomology class. Let us anticipate that Theorem A1 is an immediate corollary of the main result of this paper which shows the existence of contractible closed magnetic geodesic for levels of the energy below $c$ and positively curved in a sense which we describe in the next subsection. 
\subsection{Magnetic Curvature}
Let $SM$ be the unit sphere bundle and consider $E$ and $E^1$, the bundles of complementary and unitary complementary direction over $SM$. In particular, a fiber of $E$ and $E^1$ at the point $(p,v)\in SM$ is the set $E_{(p,v)}=\{ w \in T_pM \ | \ \langle v,w \rangle =0 \} $ and $E^1_{(p,v)}=\{ w \in S_pM \ | \ \langle v,w \rangle =0 \} $. Let $k\in \rp$ and denote by $\mathrm{Sec}$ the classical Riemannian sectional curvature. Define the \textit{k-magnetic sectional curvaure} $\Msec : E^1 \to \R$ as
\begin{equation}\label{msec}
    \Msec(v,w)=  2k\mathrm{Sec}(v,w)-\sqrt{2k} \langle(\nabla_w \Omega )(v),w \rangle+\frac{3}{4}\langle w,\Omega(v)\rangle ^2 + \frac{1}{4}|\Omega(w)|^2.  
\end{equation}
On compact oriented surfaces, \eqref{msec} coincides with the standard magnetic curvature that was introduced in the work of Paternain and Paternain \cite{gmpaternain}. In higher dimension, a similar definition of $\mathrm{Sec}^{\Omega}$ appeared in a general Hamiltonian context in the paper of Wojtkowski \cite{woj}. In particular, it can be deduced that if $\Msec$ is negative, then the magnetic flow restricted to $\Sigma_k$ is of Anosov type; see also  \cite{Gou} and \cite{Gro1} for previous results in this direction. In Lemma \ref{lemmasec} we link $\Msec$ with the second variation of a primitive of $\eta_k$ around one of its zero. \\

We point out that definition \eqref{msec} of $\Msec$ can be naturally derive from an appropriate curvature operator which encloses the classical Riemannian curvature induced by $g$ together with terms of perturbation depending on $\sigma$. Let $R$ be the Riemannian curvature operator and write $\Omega^2=\Omega \circ \Omega$. Consider the bundle operators $R^{\Omega}_k,A^{\Omega}:E \to E$ given by
\begin{equation}\label{Rmag}
R^{\Omega}_k(v,w)= 2kR(w,v)v-\sqrt{2k}  \left((\nabla_w \Omega )(v)-\frac{1}{2}(\nabla_v \Omega )(w) +\frac{1}{2} \langle(\nabla_v \Omega )(w),v\rangle v \right) , 
\end{equation}
\begin{equation}\label{operatorA}
    A^{\Omega}(v,w)=\frac{3}{4}\langle w, \Omega(v) \rangle \Omega(v) -\frac{1}{4}\Omega^2(w)- \frac{1}{4} \langle \Omega(w),\Omega(v) \rangle v     .
\end{equation}
Define the \textit{magnetic curvature operator} $\Mo:E \to E$ at the energy level $k$ as:
\begin{equation} \label{mco}
    \Mo(v,w)=R^{\Omega}_k(v,w)+A^{\Omega}(v,w)  .
\end{equation}
Observe that $M^{\Omega}_k(v,\cdot)$ is unique symmetric operator on $E_{(p,v)}$ having $\mathrm{Sec}^{\Omega}_k(v,\cdot)$ as associated quadratic form. This last remark is deduced through a straightforward computation with the help of the standard identity
\begin{equation*}
\mathrm{d}\sigma(v,w,z) = \langle \left(\nabla_{w}\Omega\right)(v), z \rangle + \langle \left(\nabla_{z}\Omega\right)(w), v \rangle + \langle \left(\nabla_{v}\Omega\right)(z), w \rangle = 0 .
\end{equation*} 
Let us point out that the accuracy of definition \eqref{mco} is also highlighted in the latest work of the author in collaboration with Marshall Reber and Terek \cite{AMT}, where $\Mo$ is naturally derived by approaching the linearization of equation \eqref{mg}. 
\\

In the classical way, we derive from $\Mo$ the magnetic curvature functions. Precisely, the \textit{k-magnetic sectional curvature} $\Msec$ can be now defined as
\begin{equation}\label{dmsec}
 \Msec(v,w)=\langle \Mo(v,w),w \rangle.
\end{equation}
The \textit{magnetic Ricci curvature} $\Mric:SM \to \R$ at the $k$-energy level is the function given by
\begin{equation}\label{dmricc}
\Mric(v)=\trace(\Mo(v,\cdot)).
\end{equation}
If Ric denote the Riemannian Ricci curvature, we obtain for definition \eqref{dmricc} the following expression 
\begin{equation}\label{mricc}
    \Mric(v)=2k\mathrm{Ric}(v)-  \sqrt{2k} \ \trace\Big((\nabla \Omega)(v)\Big)+\trace\Big(A^{\Omega}(v,\cdot)\Big)   .
\end{equation}
In the context of the second variation of the action, a magnetic Ricci curvature defined as the trace of the operator $R^{\Omega}_k$ was considered by Bahri and Taimanov in \cite{TB}. To the author's best knowledge, this is the first time the magnetic curvature and magnetic curvature functions have been defined in full generality.
\\

In this paper we use this new definition of magnetic curvature to approach the problem of finding close magnetic geodesic with energy $k$. In Lemma \ref{Bonnet-Myers} we prove in the context of magnetic systems a Bonnet-Myers theorem stating that if $\gamma$ is a zero of $\eta_k$ and $\Mric>0$, then a bound on the Morse index of $\gamma$ implies a bound on its period. This result together with index estimates for zeroes of $\eta_k$ obtained by the Struwe monotonicity argument argued in Lemma \ref{index1} and Lemma \ref{index2}, allow us to recover a converging Palais-Smale sequence and show the existence of a contractible closed magnetic geodesic. Before stating the main theorem, we remark that techniques to recover Palais-Smale through index estimates were previously considered for instance by Bahri and Coron in \cite{bagia}, where they were approaching the Kazdan-Warner problem; by Coti Zelati, Ekeland and Lions in \cite{coti} in a general variational setting and by Benci and Giannoni in \cite{begi} in the context of billiards.
\begin{teorA}\label{TeoremA}
    Let $(M,g,\sigma)$ be a magnetic system. If $k\in (0,c)$ is such that $\Mric>0$, then there exists a contractible closed magnetic geodesic with energy $k$.
\end{teorA} 
With a different approach, a weaker result was established by Bahri and Taimanov in \cite{TB}. By assuming the trace of $R^{\Omega}_k$ to be positive, the authors showed the same statement for exact magnetic systems. It is crucial that they did not consider $A^{\Omega}$ in their definition of magnetic curvature. Indeed, in Lemma \ref{aomega} we prove that the trace of $A^{\Omega}$ is always non negative and it is equal to zero at $p\in M$ if and only if $\sigma_p=0$. In accord with \eqref{mricc}, this last consideration enlarges the number of situations where we can apply Theorem A and it has a strong impact when we look at magnetic systems on low energy levels. In particular, in Proposition \ref{cor1}, we show that if the magnetic form is symplectic, then $\Msec >0$ is positive for $k$ close to zero; analogously if $\sigma$ is nowhere vanishing, then $\Mric$ is positive for values of $k$ close to zero. Intuitively we can look at symplectic magnetic systems and nowhere vanishing magnetic systems with small energies as the magnetic analogue of positively sectionally curved and positively Ricci curved manifolds in Riemannian geometry. Observe that Theorem A and Proposition \ref{cor1} immediately implies Theorem A1 stated in the previous section. \\ 

We conclude the introduction by evidencing some topological restrictions that appear when we require $\Msec$ to be positive. For instance, in dimension 2, a symplectic magnetic system can be fully characterized in terms of positive magnetic sectional curvature for low energies.
\begin{teorB}
Let $(M,g,\sigma)$ be a magnetic system on a closed oriented surface. If there exists a positive real number $k_0$ such that $\Msec>0$ for every $k \in (0,k_0)$, then either $\sigma$ is symplectic, or $\sigma=0$ and $g$ has positive Gaussian curvature  (and $M=S^2$).
\end{teorB}
The proof of Theorem B is purely topological and the reader can consult it independently from the rest of the paper in Section \ref{surfaces}. In General, by assuming $\Msec$ being positive, we also establish some topological restriction when $k>c$ and $M$ is a closed oriented even dimensional manifold, and for arbitrary $k$ when $M$ is a closed oriented surfaces and the magnetic form is exact. 
\begin{teorC}
Let $(M,g,\sigma)$ be a magnetic system on an even-dimensional oriented manifold. If $\Msec>0$ for some $k>c$, then $M$ is simply connected. Moreover, if $M$ is a compact oriented surface, $\sigma$ is exact and $\Msec>0$, then $M=S^2$ and $k\geq c$.
\end{teorC}
The proof of Theorem C is based on an adaptation of the classical Synge's theorem \cite{kli3}. Recall that this asserts that on even dimensional oriented Riemannian manifold with positive sectional curvature, no closed geodesic is length minimizer. We prove the same result for magnetic systems in Lemma \ref{Klinglemma}, where $\Msec$ replaces the Riemannian sectional curvature and closed magnetic geodesics with energy $k$ replace standard closed geodesics. Thus, the first part of the statement of Theorem C is direct consequence of the fact that if $M$ is oriented and $k>c$, then there exists a local minimizer of $\eta_k$ on each non trivial free homotopic class of $M$ \cite{Abb2}. Similarly, in the case of exact magnetic forms on closed oriented surface, the closed magnetic geodesic with energy $k$ obtained in the previously cited \cite{AsMaz,CMP,taim2} is also minimizer of $\eta_k$ for every $k\in \rp$ if $M$ is non simply connected, and for $k<c$ in the case of $S^2$.

\subsubsection*{Plan of the paper}
We present here an outline of this paper. In Section \ref{pre} we introduce the magnetic action form $\eta_k$, the Hessian of a local primitive and the notion of vanishing sequence and Morse index of a vanishing point. In particular, in Lemma \ref{chart} we build up a local chart centered on a zero of $\eta_k$ which isolates the negative directions of the Hessian; such a construction is useful for the index estimates of Lemma \ref{index1} and Lemma \ref{index2}. In the first part of Section \ref{lowenergy} we discuss the positivity of $\Msec$ and $\Mric$ when $k$ is close to zero while the second part is independently devoted to the proof of Theorem B (\ref{surfaces}). In Section \ref{hessian} we link in Lemma \ref{lemmasec} the Hessian of the magnetic form with the magnetic curvature. Here, we generalize to the magnetic context the Synge's Theorem in Lemma \ref{Klinglemma}, the Bonnet-Myers theorem in Lemma \ref{Bonnet-Myers}, and we prove Theorem C. Section \ref{proofA} is dedicated to the proof of Theorem A (and Theorem A1); we approach it by considering separately the weakly exact case in Section \ref{weaklyexact} and the non weakly exact case in Section \ref{nonweaklyexact}. The scheme of the proof is based for both cases on the following steps: first we present the minimax geometry of $\eta_k$ over the space of loops with short length. Then, in Lemma \ref{index1} and Lemma \ref{index2}, we establish index estimates for zeroes of $\eta_k$ originating from the Struwe monotonicity argument. Through these index estimates and Lemma \ref{Bonnet-Myers}, we recover a converging vanishing sequence of $\eta_k$ which concludes the proof.

\subsubsection*{Acknowledgements} The author warmly thanks his mentors Peter Albers and Gabriele Benedetti for their support during the draft of this article. He is also grateful to Alberto Abbondandolo, Luca Asselle, Johanna Bimmermann, Samanyu Sanjay and Iskander Taimanov for valuable discussions, and to James Marshal Reber and Ivo Terek for their precious corrections and suggestions on this work. A special thanks goes to Marco Mazzucchelli who suggested how to adapt the index estimate to the magnetic form in Lemma \ref{index1}. This work had been completed at the University of Toronto and the author is deeply grateful to Jacopo De Simoi for his hospitality. Finally, the author thanks the anonymous referees for their precious suggestions and for their patience in catching the inaccuracies of the first draft.
\\

The author is partially supported by the Deutsche Forschungsgemeinschaft under Germany’s Excellence Strategy EXC2181/1 - 390900948 (the Heidelberg STRUCTURES Excellence Cluster), the Collaborative Research Center SFB/TRR 191 - 281071066 (Symplectic Structures in Geometry, Algebra and Dynamics), the Research Training Group RTG 2229 - 281869850 (Asymptotic Invariants and Limits of Groups and Spaces) and by the NSERC Discovery Grant RGPIN-2022-04188.

\section{Preliminaries} \label{pre}
\subsection{Hilbert space of loops with free period and the magnetic action form}
Let $\Lambda=H^1(S^1,M)$ be the set of absolutely continuous loops on $M$ parametrized over the unit circle with $L^2$-integrable tangent vector. The space $\Lambda$ admits a structure of \textit{Hilbert  manifold} modeled as $H^1(S^1,\R^n)$. For more details and properties we refer the reader to the beautiful book of Klingenberg \cite{Kli1}. For simplicity, we assume that $M$ is orientable which, up to taking a double cover of $M$, is no loss of generality for the following constructions. The tangent space at $x\in \Lambda$ is naturally identified with the space of absolutely continuous vector field along $x$ with $L^2$-integrable covariant derivative and it is naturally isomorphic to $H^1(S^1,\R^n)$. We endow $\Lambda$ with a metric $g_{\Lambda}$ which in each tangent space, by choosing a trivialization $\mathrm{\Psi}:S^1 \times \R^n \to x^*TM$ of $TM$ along $x$, is defined as
\begin{equation}
    g_{\Lambda}( \zeta_1,\zeta_2 ) = \int_0^1 \Big[ \langle \zeta_1,\zeta_2 \rangle_{x} + \Big\langle \dt \zeta_1, \dt \zeta_2 \Big\rangle_{x}  \Big]   \mathrm{d}s   .  
\end{equation}
The distance $d_{\Lambda}$ induced by $g_{\Lambda}$ gives to $\Lambda$ a structure of \textit{complete Riemannian manifold}. The space $C^{\infty}(S^1,M)$ is dense in $\Lambda$, and a local chart centered at a smooth loop $x$ is given for instance by
\begin{equation*}
    F_{\Lambda}:H^1(S^1,B_{r}(0)) \to \Lambda  , \ F_{\Lambda}(\zeta)(t)= \mathrm{exp}_{x(t)}(\mathrm{\Psi}(t,\zeta(t))) ;
\end{equation*}
where $B_r(0) \subset \R^n$ is an open ball and exp is the exponential map of $(M,g)$. As pointed out in \cite[Remark~2.2]{Abb2}, $F_{\Lambda}$ is bi-Lipschitz i.e. $F_{\Lambda}$ and $F_{\Lambda}^{-1}$ are Lipschitz with respect to the standard distance $d_0$ of $H^1(S^1,\R^n)$ and $d_{\Lambda}$. This is a consequence of the fact that the norm $| \cdot |_0$ of $H^1(S^1,\R^n)$ induced by the Euclidean scalar product of $\R^n$ and the norm $| \cdot |_{\Lambda}$ induced by $(F_{\Lambda})^{*} g_{\Lambda}$ are equivalent on $H^1(S^1,B_r(0))$. \\

Consider now $\M= \Lambda \times \rp$. This set can be interpreted as the set of absolutely continuous loops on $M$ with $L^2$-integrable tangent vector and free period. Indeed, a point $(x,T) \in \M$ is identified with $\gamma :[0,T] \to M$ through $\gamma(t)=x(\frac{t}{T})$. Vice versa, an absolutely continuous loop $\gamma$ with $L^2$-integrable tangent vector corresponds to the pair $(x,T)$ with $x(s)=\gamma(sT)$. We provide $\M$ with the Riemannian structure obtained from the product of $\Lambda$ endowed with $g_{\Lambda}$ and the Euclidean structure of $\rp$. In particular, the tangent space splits into 
\begin{equation}
    T\M = T\Lambda \oplus \R \frac{\partial}{\partial T}  ,
\end{equation}
and in this splitting the metric is given by 
\begin{equation*}
    g_{\M} = g_{\Lambda} + \mathrm{d}T^2  .
\end{equation*}
If $x\in C^{\infty}(S^1,M)$, a local chart centered at $(x,T)$ is obtained through the product
\begin{equation}\label{chartM}
    F_{\M}=F_{\Lambda} \times \mathrm{Id}_{\rp}: H^1(B_r(0),M)\times \rp \to \M  ,  \ F_{\M}(\zeta, \tau)=(F_{\Lambda}(\zeta),\tau)  .
\end{equation}
With abuse of notation, write $| \cdot |_{0}$ the standard product norm of $H^1(S^1,\R^n)\times \rp$ and observe that $| \cdot |_{0}$ and the norm induced by $(F_{\M})^*g_{\M}$ are equivalent on $H^1(B_r(0),M)\times \rp$. In analogy with $F_{\Lambda}$, this implies that the local chart $F_{\M}$ defined in \eqref{chartM} is bi-Lipschitz since it is the product of two bi-Lipschitz maps. 
The distance induced by $g_{\M}$ is not complete because $\mathrm{d}T^2$ is not complete in the Euclidean factor. Nevertheless, completeness is obtained by restricting $g_{\Lambda}$ on sets of the form $\Lambda \times [T_*, +\infty)$ for every arbitrary $T_*>0$. Finally, since $\M$ is homotopically equivalent to $\Lambda$, its connected components are in correspondence with the elements of $[S^1,M]$, namely the set of conjugacy classes of $\pi_1(M)$. In particular, with $\M_0$ we indicate the connected component of contractible loops with free period. \\

Hereafter, we often identify $\gamma$ with the respective $(x,T)$. For simplicity, we also use the``dot" to indicate the Levi-Civita covariant derivative ``$\frac{D}{\mathrm{d}{t}}$". A vector field over $x$ and its respective parametrization over $\gamma$ is denoted with the same symbol. Observe that the rescaling of the tangent vectors and their time derivatives with respect to the two different parametrizations is given by the identity $\speed(t) = \frac{1}{T}  \dot{x}(\frac{t}{T}) $, and by $\dot{V}(t) = \frac{1}{T} \dot{V}(\frac{t}{T})$ if $V$ is a vector field along $x$.

\subsection{Magnetic action form}
For $k\in \rp$ consider the $k$-kinetic action $A_k: \M \to R$ defined as: 
\begin{flalign*}
    A_k(x,T)&= T\int_0^1 \Big\{ \frac{|\dot{x}|^2}{2T^2} +k \Big\} \mathrm{d}s \\
    &= \int_0^T \Big\{ \frac{|\speed|^2}{2}+k \Big\}  \mathrm{d}t  .
\end{flalign*}
Let $\Theta \in \Omega^1(\Lambda)$ be such that 
\begin{equation*}
    \Theta_x(V)=\int_0^1 \langle \Omega(V),\dot{x} \rangle  \mathrm{d}s.
\end{equation*}
Denote by $\pi_{\Lambda}: \M \to \Lambda$ the projection into the first factor of $\M$. The \textit{$k$-magnetic action form} $\eta_k\in \Omega^1(\M)$ is defined by:
\begin{equation*}
    \eta_k(\gamma) = d_{\gamma}A_k + \pi_\Lambda ^* \Theta_x  .
\end{equation*}
If $x$ is of class $C^2$ and $(V,\tau)\in T_x\Lambda \oplus \R \frac{\partial}{\partial T}$, then $\eta_k$ acts as follows: 
\begin{equation*}
    \eta_k(\gamma)(V,0)=-\int_0^T \langle \nabla_{\speed}\speed - \Omega(\speed),V \rangle  \mathrm{d}t   , \  \eta_k(\gamma)(0,\tau)= \frac{\tau}{T}\int_0^T \Big\{ k-\frac{|\speed|^2}{2} \Big\} \mathrm{d}t  . 
\end{equation*}
In agreement with \cite[Lemma~2.2]{AB} and \cite[Lemma~3.1]{Abb2},  $\eta_k$ is a smooth a section of $T\M$ and $\eta_k(\gamma)=0$ if and only if $\gamma$ is a closed magnetic geodesic with energy $k$. We write $\Z(\eta_k)$ the zero set of $\eta_k$. An important property of the magnetic action form concerns its behaviour when integrated over a loop of $\M$. Indeed if $u:[0,1] \to \M$ is a closed path, the integral $\int_0^1 u^*\eta_k$ depends only on the homotopy class of $u$. In this sense $\eta_k$ is said to be \textit{closed} \cite[Corollary~2.4]{AB}. \\ 

 Let $\U \subseteq \M $ be an open set diffeomorphic to a ball and $(x_0,T_0) \in \U$. A local primitive $\prim : \U \to \R$ of $\eta_k$, centered in $(x_0,T_0)$, is defined by
\begin{equation}\label{primitive}
 \prim(x,T) = A_k(x,T)+ \int_{S^1 \times [0,1]} c_{x_0,x}^* \ \sigma  .
\end{equation}
Here, $c_{x_0,x}: S^1 \times [0,1] \to M$ is a cylinder which connects $x_0$ and $x$ such that $c_{x_0,x}( \cdot , s) \in \pi_{\Lambda}(\U)$ for every $s$. Indeed, the closedness of $\eta_k$ makes the above definition independent from the choice of $c_{x_0,x}$. Because $d\prim = \eta_k$ on $\U$, it is obvious that $\gamma \in \Z(\eta_k)$ if and only if $\gamma$ is a critical point for every local primitive of $\eta_k$.

\subsection{Hessian of $\eta_k$ and the index of a vanishing point}
Let $\gamma$ be a zero of $\eta_k$ and $\prim$ be a local primitive of $\eta_k$ centered at $\gamma$. If $\zeta \in T_{\gamma} \M$, we naturally identify the tangent space $T_{\zeta}(T_{\gamma} \M)$ with $T_{\gamma}\M$ itself. Under this identification we can look at the second (Fr\'echet) derivative $ \mathrm{d}^2_{\gamma} \prim $ of $\prim$ at the point $\gamma$, as a bilinear form on $T_{\gamma}\M$. Because the connection on $M$ is the Levi-Civita connection, such a bilinear form is symmetric. The Hessian of $\prim$ at $\gamma$, denoted by $\mathrm{Hess}_{\gamma}(\prim)$, is the quadratic form associated with $ \mathrm{d}^2_{\gamma} \prim $. We refer the reader to \cite{chang} for a general overview. The fact that two local primitives of $\eta_k$ differ by a constant makes natural the following definition.
\begin{defn}\label{hessian} 
Let $\gamma \in \Z(\eta_k)$. The Hessian of $\eta_k$ at $\gamma$, denoted by $\hess$, is the quadratic form defined as
\begin{equation*}
    \hess=\mathrm{Hess}_{\gamma}(\prim)  ,
\end{equation*}
where $\prim$ is any primitive of $\eta_k$ on a neighbourhood of $\gamma$.
\end{defn}
\begin{lemma}\label{secondvariation}
    Let $\gamma \in \Z(\eta_k)$ and $\zeta=(V,\tau)\in T_{\gamma}\M$. Then
\begin{equation*}
  \begin{aligned}
    \hess(V,\tau) =& \int_0^T \left\{ \langle \dot{V}-\Omega(V),\dot{V}\rangle-\langle R(V,\speed)\speed-\nabla_V \Omega (\speed),  V\rangle \right\}   \mathrm{d}t -\int_0^T \frac{\langle \dot{V},\speed \rangle^2}{|\speed|^2}  \mathrm{d}t  \\
 & \ \ \ \ \ \ \ \ \ \ \ \ \ \ \ \ \ \ \ \ \ \ \ \ \ \ \ \ +\int_0^T \left( \frac{\langle \dot{V},\speed \rangle}{|\speed|}-\frac{\tau}{T}|\speed| \right)^2   \mathrm{d}t  .
\end{aligned}
\end{equation*}
\begin{proof}
    The computation follows by adopting \cite[Section~2]{Gou} to a primitive of $\eta_k$ as given by \eqref{primitive}. 
    
\end{proof}
\end{lemma}
If $\gamma$ is a vanishing point, the signature of $\hess$ is well-defined because it is independent of the local coordinate chart. Therefore, there exist two vector subspaces $H_{\gamma},E_{\gamma} \subset T_{\gamma} \M$ such that the tangent space splits into
\begin{equation*}
T_\gamma \M = H_{\gamma} \oplus E_{\gamma}  ,
\end{equation*}
and $E_{\gamma}$ is the maximal subspace, in terms of dimension, for which the restriction $\hess|_{_{E_{\gamma}}}$ is negative definite.
\begin{defn} Let $\gamma \in \Z(\eta_k)$. The \textit{Morse index} of $\gamma$, denoted by $\mathrm{index}(\gamma)$, is the non negative integer given by
\begin{equation*}
\mathrm{index}(\gamma)=\mathrm{dim}E_{\gamma}  . 
\end{equation*}
\end{defn}
As argued in \cite[Proposition~3.1]{Abb1}, the self adjoint operator associated with $ \mathrm{d}^2_{\gamma} \prim $ is a compact perturbation of a positive Fredholm operator, which implies that the index of a vanishing point of $\eta_k$ is always finite. \\ 

We conclude the section by constructing a system of local coordinates around a zeroes of $\eta_k$ which isolates the direction of $T_{\gamma}\M$ where $\hess$ is negative definite. Such construction is used in the proof of Lemma \ref{index1} and Lemma \ref{index2} in Section \ref{proofA}. \\ 

Let $\gamma$ such that $\mathrm{index}(\gamma)=d$ and consider $(\U_{\gamma},F_{\M})$ local coordinates centered at $\gamma$ as given in Equation~\eqref{chartM}. By the assumption on the index and by the equivalence of the norms $| \cdot |_0$ and $| \cdot |_{\M}$ there exist two vector spaces $H_{\gamma},E_{\gamma} \subset H^1(S^1,\R^n)\times \rp$ and a positive constant $D_{\gamma}$ such that 
\begin{equation*}
H^1(S^1,\R^n)\times \rp= H_{\gamma} \oplus E_{\gamma}  , 
\end{equation*}
 the space $E_{\gamma}$ is of dimension $d$ and
\begin{equation}\label{inequality}
    \lochess(\zeta,\tau)\leq -4D_{\gamma}|(\zeta,\tau)|^2_{0}  ,  \ \ \forall (\zeta,\tau) \in E_{\gamma}  .
\end{equation}
Hereafter, we denote by $B_r^H$ and $B^E_r$ respectively the open ball of radius $r$ centered at the origin of $H_{\gamma}$ and $E_{\gamma}$.
\begin{lemma}\label{chart}
    There exists a local chart $(\mathrm{\Phi}_{\gamma}^{-1}(\V_{\gamma}),\mathrm{\Phi}_{\gamma})$ centered at $\gamma$ such that $\V_{\gamma}=B_r^H \times B_r^E$. Moreover, this system of local coordinates enjoys the following properties:
    \begin{itemize}
    \item[(i)]  $\mathrm{\Phi}_{\gamma}$ is bi-Lipschitz (with respect to $d_0$ and $d_{\M}$) and $\mathrm{\Phi}_{\gamma}(\Z(\eta_k))\subseteq B_r^H \times \{ 0 \}$;
    \item[(ii)] if $\prim$ is an arbitrarily primitive of $\eta_k$ on $\mathrm{\Phi}_{\gamma}^{-1}(\V_{\gamma})$, then for every $(y_h,y_e)\in B_r^H \times B_r^E$ and for every $0<\lambda< r - |y_e|_{_0}$ it holds
        \begin{equation*}
       \locprim \Big(y_h,y_e+\lambda \frac{y_e}{ \ |y_e|_{_{0}}}\Big) \leq \locprim \Big(y_h,y_e \Big)-D_{\gamma}\lambda^2  . 
      \end{equation*}
      \end{itemize}

\begin{proof}
Let $(\U_{\gamma},F_{\M})$, $H_{\gamma}$ and $E_{\gamma}$ as above and $\prim$ a local primitive defined on $\U_{\gamma}$. By \refeq{inequality} and by the smoothness of $\prim$ there exists $\V_{\gamma}\subset U_{\gamma}$ open neighborhood of $\gamma$ such that

\begin{equation}\label{inequality2}
    \Q_{\theta}(F_{\M}^* \eta_k)(\zeta,\tau)\leq - 2D_{\gamma}|(\zeta,\tau)|_{0}^2 ,  \ \forall \theta \in F_{\M}^{-1}(V_{\gamma}) \ , \ \forall (\zeta,\tau) \in E_{\gamma}  .
\end{equation}

Without loss of generality, for a positive real $R$, we can assume $\V_{\gamma}$ being diffeomorphic through $F_{\M}$ to $B_R^H \times B^E_R$ with coordinates $(y_h,y_e)$ centered in $\gamma$. 
Consider now the function $G_{\gamma}:B_R^H \times B^E_R \to E_{\gamma}$ defined as 
\begin{equation*}
    G_{\gamma}(y_h,y_e)=\partial_{y_e}\locprim(y_h,y_e)  .
\end{equation*}
As consequence of the Implicit Function Theorem, there exists a function $g_{\gamma}:B_R^H \to B_R^E$ such that, over $B_R^H \times B^E_R$, the equality $G_{\gamma}(y_h,y_e)=0$ holds if and only if $y_e=g_{\gamma}(y_h)$. In particular, all the critical points of $\prim$ in this local coordinate chart belong to the graph of $g_{\gamma}$ which we denote by $\mathcal{G}(g_{\gamma})$. Consider the map $T(y_h,y_e)=(y_h,y_e-g_{\gamma}(y_h))$ which pointwise translates $\mathcal{G}(g_{\gamma})$ to $B_R^H \times \{ 0 \}$ into the direction of $y_e$. Let $r\in(0,R)$ such that $B_r^H \times B_r^E$ is the image through $T$ of an open neighbourhood of $\mathcal{G}(g_{\gamma})$ and set $\mathrm{\Phi}_{\gamma}=F_{\M} \circ T$ defined on $\V_{\gamma}=\mathrm{\Phi}_{\gamma}^{-1}(B_r^H \times B_r^E)$. Observe that $\mathrm{\Phi}_{\gamma}$ is bi-Lipschitz because is a composition of bi-Lipschitz maps. Moreover, if $\theta \in \Z(\eta_k)$, then $\mathrm{\Phi}_{\gamma}(\theta)\in B_R^H \times \{ 0 \} $ since $F_{M}$ sends zeroes of $\eta_k$ to $\mathcal{G}(g_{\gamma})$ and $T$ sends the graph of $g_{\gamma}$ to $B_r^H \times \{ 0 \}$.

Finally, let  $\lambda<r-|y_e|_{0}$ and for simplicity write $q(t)=\Big(y_h,y_e+t\lambda \frac{y_e}{|y_e|}_0 \Big)$. Up to shrinking $r$ we have that

\begin{flalign*}
    \locprim(q(1))-\locprim(q(0))&=\int_0^1\frac{\mathrm{d}}{\mathrm{d}t}\locprim(q(t)) \mathrm{d}t \\
    &=\int_0^1\mathrm{d}_{q(t)}\locprim(\dot{q}(t)) \mathrm{d}t \\
   \  \ \ \ \ \ \ \ \ \ &=\int_0^1\int_0^t \mathrm{d}\locprim_{_{(y_h,y_e)}}(\dot{q}(0)) \mathrm{d}s \mathrm{d}t \\
    & \ \ \ \ \ \ \ \ \ \ \ \ \ \ \ \ \ + \int_0^1\int_0^t \left[\frac{\mathrm{d}}{\mathrm{d}s}d\locprim_{q(t)}(\dot{q}(0)) \right] \mathrm{d}s \mathrm{d}t \\
     &\leq  \int_0^1\int_0^t  \mathrm{Hess}_{q(t)}(S^k_{\sigma}\circ \mathrm{\Phi_{\gamma}^{-1}})(\dot{q}(0)) \mathrm{d}s \mathrm{d}t \\
    &\leq -D_{\gamma} \lambda^2 ,
\end{flalign*}

where in the above inequalities we first use the fact that the first derivative of $\locprim$ with respect to $y_e$ is non positive for points close to  $B_r^H \times \{ 0 \}$ and then the inequality pointed out in~\eqref{inequality2} 

\end{proof}
\end{lemma}

\subsection{Vanishing sequences} 
Our goal is to find zeroes of $\eta_k$ via minimax methods. The next definition generalizes to the magnetic action form the notion of Palais-Smale sequence.
\begin{defn}
    A $vanishing \ sequence$ for $\eta_k$ is a sequence $(x_n,T_n)=(\gamma_n)$ of $\M$ such that
    \begin{equation*}
        |\eta_k(\gamma_{n})|_{_{\M}} \to 0  .
    \end{equation*}
\end{defn}
Because $\eta_k$ is continuous, $\Z(\eta_k)$ coincides with the limit point set of vanishing sequence. Thus understanding whether a vanishing sequence admits or not a converging subsequence becomes of crucial importance. Generally if the sequence of periods $T_n$ diverges to infinity or tends to zero there is no hope to find a limit point for $\gamma_n$. Nevertheless, in \cite[Theorem~2.6]{AB}, it is established that this is exactly the case to avoid in order to have compactness on a vanishing sequence. 
\section{Magnetic curvature on low energy levels}\label{lowenergy}
\subsection{Symplectic and nowhere vanishing magnetic forms.}\label{lowenergysub}
A question which emerges naturally is whereas a magnetic system show up a k-level of the energy positively or negatively curved in terms of $\Msec$ or $\Mric$. We are interested into magnetic system positively curved and in a range of energy close to zero. This requirement is satisfied by two standard classes of magnetic systems: symplectic magnetic systems with respect to the magnetic sectional curvature and nowhere vanishing magnetic systems with respect to the magnetic Ricci curvature. We need the following relevant lemma.
\begin{lemma}\label{aomega} Let $A^{\Omega}$ the operator defined in~\eqref{operatorA}. Then the following statements hold
\begin{itemize}
\item[(i)] $ \langle  A^{\Omega}(v,w) , w  \rangle \geq 0$ for every $(v,w)\in E^1$ and it is equal to zero if and only if $\Omega(w)=0$,
\item[(ii)] $\mathrm{trace}\Big(A^{\Omega}(v,\cdot)\Big) \geq 0$ for every $v\in SM$ and it is equal to zero if and only if $\Omega_p=0$.
\end{itemize}
\begin{proof} Observe that
\begin{equation}\label{A2}
\langle A^{\Omega}(v ,w), w \rangle = \frac{3}{4} \langle w, \Omega(v) \rangle^2 + \frac{1}{4}|\Omega(w)|^2  ,
\end{equation}
which immediately implies point (i). Moreover, if we extend $v$ to an orthonormal basis $\{ v, e_2, ..., e_n \} $ then by \eqref{A2} we get
\begin{align}
\trace\Big(A^{\Omega}(v, \cdot)\Big)&= \sum_{i=2}^n \langle A^{\Omega}(v ,e_i), e_i \rangle \nonumber \\
& =\sum_{i=2}^n \Big\{  \frac{3}{4} \langle \langle e_i ,\Omega(v) \rangle \Omega(v), e_i \rangle + \frac{1}{4} \langle \Omega(e_i),\Omega(e_i) \rangle \Big\} \nonumber &&\\
&=\sum_{i=2}^n \frac{3}{4} \langle e_i ,\Omega(v) \rangle^2 + \frac{1}{4}\Big(\sum_{i,j=2}^n \langle \Omega(e_i),e_j \rangle^2 + \sum_{i=2}^n \langle \Omega(e_i),v \rangle^2 \Big) \nonumber &&\\
&=\sum_{i=2}^n  \langle e_i ,\Omega(v) \rangle^2 + \frac{1}{4} \sum_{i,j=2}^n \langle \Omega(e_i),e_j \rangle^2  . \label{terms}
\end{align}
which is always non negative and equal to zero if and only if $\Omega_p = 0$. Thus, also point (ii) holds.

\end{proof}
\end{lemma}

\begin{pro}\label{cor1} Let $(M,g,\sigma)$ be a magnetic system. The followings hold:

\begin{itemize}
\item[(i)] if $\sigma$ is symplectic then there exists a $k_0>0$ such that $\Msec>0$ for every $k\in (0,k_0)$;
\item[(ii)] if $\sigma$ is nowhere vanishing then there exists a $k_0>0$ such that $\Mric>0$  for every $k\in (0,k_0)$. 
\end{itemize}

\begin{proof}
If $\sigma$ is symplectic then, by the compatibility condition~\eqref{compa}, it follows that $\Omega(w)\neq 0$ for every $(p,w)\in TM$. In particular, \eqref{A2} implies that $\langle A^{\Omega}(v ,w), w \rangle$ is strictly positive for every $(v,w) \in E^1$. Analogously, if $\sigma$ is nowhere vanishing, then at least one term in \eqref{terms} is different from zero which implies that the trace of $A^{\Omega}$ is strictly positive for every $v \in SM$. Since $E^1$ and $SM$ are compact and since the operator $A^{\Omega}$ does not depend from $k$, if $\sigma$ is symplectic then for small values of $k$ it follows that
\begin{equation*}
\Msec = \langle R^{\Omega}_{k}(v,w), w \rangle + \langle A^{\Omega}(v,w),w \rangle>0  , \ \forall (v,w) \in E^1  .
\end{equation*}
If $\sigma$ is nowhere vanishing, then for small values of $k$, it also holds that
\begin{equation*}
\Mric = \mathrm{trace}\Big(R^{\Omega}_k(v,\cdot) \Big) + \mathrm{trace}\Big(A^{\Omega}(v,\cdot)   \Big)>0 \ , \ \forall v\in SM  .
\end{equation*}
The statement follows. 

\end{proof}
\end{pro}
\subsection{Symplectic magnetic forms in dimension 2.} \label{surfaces}
In this subsection $(M,g)$ is a closed oriented Riemannian surface and we denote by $\K:M \to \R$ its Gaussian curvature. If $\sigma$ is a closed two form on $M$, there exists a smooth function $b:M \to \R$ such that 
\begin{equation}\label{volume}
\sigma = b  \mu_g  ,
\end{equation}
 where $\mu_g$ is the volume form on $M$ induced by the metric $g$. For simplicity we refer to a magnetic form $\sigma$ through the function $b$ given in~\eqref{volume}. Observe that, in any orthonormal basis the operator $\Omega$ enjoys the following identities
 \begin{equation}\label{ope}
  \Omega = b  \mathrm{J}  \ \mathrm{and} \  \nabla_v \Omega = \mathrm{d}b(v)\mathrm{J}  ,
 \end{equation}
 where $\mathrm{J}$ denotes the rotation of an angle $\frac{\pi}{2}$ in accord with the fixed orientation.
 Due to dimensional reasons, $\mathrm{Ric}_k^b$ coincides with $\mathrm{Sec}_k^{b}$ and its expression in terms of $b$ is pointed out by the next lemma.
 \begin{lemma} Let $(M,g,b)$ be a magnetic system on a closed oriented surface. Then
 \begin{equation*}
 \mathrm{Sec}^b_k(v) = 2k \K - \sqrt{2k} \mathrm{d}b\left(\mathrm{J} v\right) + b^2 ,
 \end{equation*}
 \begin{proof}
Fix $\{ v , e_2= \mathrm{J} \cdot v \}$ as orthonormal basis. By definition \eqref{msec} and by \eqref{ope} we obtain that
 \begin{align*}
 \mathrm{Sec}^b_k(v)& = 2k \mathrm{Sec}(v,e_2) - \sqrt{2k}\langle (\nabla_{e_2} \Omega )(v) , e_2 \rangle + \langle A^{\Omega}(v,e_2),e_2 \rangle \\
 & = 2k  \K - \sqrt{2k}\langle \mathrm{d}b(e_2) \mathrm{J} v , e_2 \rangle + \frac{3}{4}b^2 + \frac{1}{4}b^2 \\ 
 & = 2k  \K - \sqrt{2k}  \mathrm{d}b(\mathrm{J} v) + b^2  . 
 \end{align*}
 \end{proof}
 \end{lemma}
In this setting, Theorem B is reformulated as follows
\begin{teorB}  Let $(M,g,b)$ be a magnetic system on a closed oriented surface. If there exists a positive real constant $k_0$ such that $ \mathrm{Sec}^b_k > 0$ for every $k \in (0,k_0)$, then either $b$ is nowhere zero or $b$ is constantly zero and $\K>0$ (and $M=S^2$).
\begin{proof}
Suppose that for small values of $k$ we have that $\mathrm{Sec}^b_k>0$. Denote by $S$ the subset of $M$ defined as
\begin{equation*}
S = \{ p \in M \ | \ b(p)=0 \}  .
\end{equation*}
This subset is closed by definition. If we show that $S$ is also open, the statement follows. Let $p \in S$ and for an arbitrary small radius $r$, let $B_r(p)$ be an open ball of $M$ centered at $p$. Let $q \in B_r(p)$ be such that 
\begin{equation*}
|b(q)| = \max _{z \in B_r(p)} |b(z)| \ . 
\end{equation*}
Assume that $|b(q)|\neq 0$ and denote by $|\mathrm{d}b|_{\infty}$ the uniform norm of $\mathrm{d}b$. By assumption, for small values of $k$ it yields that
\begin{equation*}
2k \K + b^2 > \sqrt{2k}|\mathrm{d}b|_{\infty} , 
\end{equation*}
which implies that
\begin{equation*}
|b(q) - b(p) |  \leq \int_0^{d(p,q)} |\mathrm{d}b|_{\infty} \mathrm{d}t \leq \frac{r}{\sqrt{2k}} ( 2k \max_{z \in B_r(p)} \K(z) + |b(q)|^2) = r \sqrt{2k} \max_{z \in B_r(p)} \K(z) + \frac{r}{\sqrt{2k}} |b(q)|^2.
\end{equation*}
Since $b(p)=0$, up to shrinking $r$ we can choose $k$ such that $\sqrt{2k}=|b(q)|$ and deduce from the above inequality that
\begin{equation*}
|b(q)| \leq (r \max_{z \in B_r(p)} \K(z) + r ) |b(q)|  . 
\end{equation*}
Up to shrinking again $r$ and $k$, it yelds that  
\begin{equation*}
r \max_{z \in B_r(p)} \K(z) + r<1  , 
\end{equation*}
 which necessary implies that $|b(q)|=0$. Therefore, for every $p\in S$ we can find an open ball $B_r(p)$ centered such that $b|_{B_r(p)}=0$ which concludes the proof.
 
\end{proof}
\end{teorB}

\section{Magnetic curvature and Hessian of $\eta_k$} \label{hessian}
Let $\gamma\in \Z(\eta_k)$ and consider the splitting $\speed \oplus \{ \speed \}^{\perp}$ of $\gamma^* TM$, where $\{ \speed \}^{\perp}$ is the orthogonal complement of $\speed$. We decompose a vector field along $\gamma$ into $V=V_1+V_2$, where $V_1$ denotes its component along $\speed$ and $V_2$ its orthogonal component. To avoid any kind of confusion we adopt the following notations: $\frac{D}{\mathrm{d}t} V_1=\dot{V_1}$ and $(\dot{V})_1=\langle \dot{V},\speed \rangle \frac{\speed}{|\speed|^2}$. Analogously, $\frac{D}{\mathrm{d}t} V_2=\dot{V_2}$ while $(\dot{V})_2$ indicate the orthogonal projection of $\dot{V}$ on $\{ \speed \}^{\perp}$. The next crucial lemma shows how such a splitting let us write the Hessian of $\eta_k$ in terms of the magnetic curvature.

\begin{lemma}\label{lemmasec}  Let $V=V_1+V_2$ be a variation along $\gamma$ and $\tau \in \R$. Then
\begin{equation}\label{crucial}
\begin{aligned}
    \hess(V,\tau)=\int_0^T  \Big|(\dot{V})_2-\frac{1}{2}(\Omega(V_1)&+\Omega(V))_2\Big|^2   -\int_0^T  |V_2|^2 \Msec\Big(\frac{\speed}{|\speed|},\frac{V_2}{|V_2|}\Big) \mathrm{d}t   \\ 
    &+\int_0^T \Big(\frac{\langle \dot{V},\speed \rangle}{|\speed|} - \frac{\tau}{T}|\speed|\Big)^2  \mathrm{d}t . 
\end{aligned}
\end{equation}

\begin{proof} Preliminarily, we point out the following identities:
\begin{equation}\label{ide1}
(\nabla_{V_1} \Omega)(\speed)=(\nabla_{\speed}\Omega)(V_1)  , 
\end{equation}
\begin{equation}\label{ide2}
     \dot{V}_1=(\dot{V})_1-\Omega(V)_1+\Omega(V_1)  , 
    \end{equation}
\begin{equation}\label{ide3}
        \frac{d}{\mathrm{d}t} \langle \Omega(V_1),V \rangle=\langle (\nabla_{\speed}\Omega)(V_1),V \rangle + \langle\Omega(\dot{V}_1),V \rangle + \langle \Omega(V_1),\dot{V} \rangle   . 
    \end{equation}

Decompose $V$ into its component $V_1$ and $V_2$ in the expression of $\hess$ in Lemma \ref{secondvariation}. By the fact that
\begin{equation*}
\Big\langle R^{\Omega}_k\left(\frac{\speed}{|\speed|},V_2\right),V_2\Big\rangle = \langle R(V_2,\speed)\speed-\nabla_{V_2} \Omega (\speed),  V_2 \rangle,
\end{equation*}
and by identity \eqref{ide1}, we first obtain that
\begin{equation}\label{conto1}
\hess(V,\tau)=\int_0^T P(V) \mathrm{d}t   -\int_0^T \Big\langle R^{\Omega}_k\left(\frac{\speed}{|\speed|},V_2\right),V_2\Big\rangle \mathrm{d}t   + \int_0^T\Big(\frac{\langle \dot{V},\speed \rangle}{|\speed|} - \frac{\tau}{T}|\speed|\Big)^2\mathrm{d}t  , 
\end{equation}
where $P(V)= |(\dot{V})_2|^2+ \langle (\nabla_{\speed}\Omega)(V_1),V \rangle - \langle \dot{V}, \Omega(V) \rangle $. With the help of \eqref{ide2}, \eqref{ide3} and a Stokes argument it yields that

\begin{flalign*}
\int_0^T P(V)\mathrm{d}t&=\int_0^T \Big\{ |(\dot{V})_2|^2 + \frac{d}{dt} \langle \Omega(V_1),V \rangle - \langle\Omega(\dot{V}_1),V \rangle - \langle \Omega(V_1),\dot{V} \rangle  - \langle \dot{V}, \Omega(V) \rangle    \Big\}  \mathrm{d}t&& \\\nonumber
&=\int_0^T \Big\{ |(\dot{V})_2|^2  + \langle \dot{V}_1,\Omega(V) \rangle - \langle \Omega(V_1), (\dot{V})_2 \rangle - \langle (\dot{V})_2, \Omega(V) \rangle - \langle (\dot{V})_1,\Omega(V) \rangle  \Big\}  \mathrm{d}t&& \\\nonumber
&=\int_0^T \Big\{ |(\dot{V})_2|^2- \langle \Omega(V_1)+\Omega(V)_2,(\dot{V})_2 \rangle+\langle \dot{V}_1-(\dot{V})_1, \Omega(V)\rangle \Big\} \mathrm{d}t&& \\\nonumber
&=\int_0^T \Big\{ |(\dot{V})_2- \frac{1}{2}(\Omega(V_1)+\Omega(V))_2|^2- H(V) \Big\}  \mathrm{d}t  , && \\\nonumber
\end{flalign*}
where by stressing the computation we write 
\begin{equation*}
H(V)=\frac{1}{4}|\Omega(V_1)+\Omega(V)_2|^2+ \langle \Omega(V)_1,\Omega(V) \rangle - \langle \Omega(V_1),\Omega(V)_2 \rangle.
\end{equation*}
Consider definition \eqref{operatorA} of $A^{\Omega}$ to point out that
 \begin{flalign*}
H(V)&= \frac{1}{4}|2\Omega(V_1)+\Omega(V_2)_2|^2+|\Omega(V_2)_1|^2+\langle \Omega(V_1),\Omega(V_1)\rangle - \langle \Omega(V_1),\Omega(V_2)_2 \rangle&&  \\\nonumber
&= \frac{1}{4}|\Omega(V_2)_2|^2+|\Omega(V_2)_1|^2  &&\\ \nonumber
&= \frac{1}{4}|\Omega(V_2)|^2+\frac{3}{4}|\Omega(V_2)_1|^2  &&\\ \nonumber
&=  \Big\langle A^{\Omega}\left(\frac{\speed}{|\speed|},V_2\right),V_2\Big\rangle . &&\\ \nonumber
 \end{flalign*}
 Therefore, we have that
\begin{equation*}
\int_0^T P(V)\mathrm{d}t = \int_0^T \Big\{  |(\dot{V})_2+ \frac{1}{2}(\Omega(V_1)+\Omega(V))_2|^2 -  \Big\langle A^{\Omega}\left(\frac{\speed}{|\speed|},V_2\right),V_2\Big\rangle \Big\}  \mathrm{d}t . 
\end{equation*}
By definition \eqref{dmsec} of $\Msec$, the statement follows with the substitution of the above identity in \eqref{conto1}. 
 
\end{proof}
\end{lemma}

The next lemma shows how we can always construct variations along $\gamma$ such that their evaluation in $\hess$ has no terms which depend on the tangential component.
\begin{lemma}\label{tecnica1}
Let $V$ be a variation along $\gamma$ such that $\langle V,\speed \rangle=0$. Then there exists a periodic function $g:[0,T] \to \R$ and a real constant $\tau$ depending linearly on $V$ such that, if we write $W=V+g\speed$, then 
\begin{equation}
\hess(W,\tau)= \int_0^T  \Big|(\dot{V})_2-\frac{1}{2} \Omega(V)_2\Big|^2  \mathrm{d}t  -\int_0^T |V|^2\Msec\Big(\frac{\speed}{|\speed|},\frac{V}{|V|}\Big) \mathrm{d}t  . 
\end{equation}
\begin{proof}
Consider $g$ and $\tau$ defined as follows
\begin{equation*}
g(t)=-\int_0^t \Big\{ \frac{\langle \dot{V},\speed \rangle}{|\speed|^2} - \frac{\tau}{T} \Big\} \mathrm{d}t \ \ , \ \  \tau= \int_0^T \frac{\langle \dot{V},\speed \rangle}{|\speed|^2} \mathrm{d}t  .
\end{equation*}
In particular, the couple $(g,\tau)$ satisfies the differential problem along $\gamma$ given by
\begin{equation}\label{differentialproblem}
\begin{cases}
    \dot{g}+\frac{\langle \dot{V},\speed \rangle}{|\speed|^2} - \frac{\tau}{T}=0 \\
    g(0)=g(T)=0
    \end{cases}
\end{equation}
Thus if $W=V+g\speed$, then
\begin{equation}\label{usa1}
 \Big(\frac{\langle \dot{W},\speed \rangle}{|\speed|} - \frac{\tau}{T}|\speed|\Big)^2 =  |\speed|^2\Big(\frac{\langle \dot{V},\speed \rangle}{|\speed|^2}+\dot{g} - \frac{\tau}{T}\Big)^2 =0  , 
\end{equation}
and
\begin{equation}\label{usa2}
(\dot{W})_2-\frac{1}{2}(\Omega(W_1)+\Omega(W))_2 =(\dot{V})_2+g\ddot{\gamma}-g\Omega(\speed)-\frac{1}{2}\Omega(V)_2=(\dot{V})_2-\frac{1}{2}\Omega(V)_2  .
\end{equation}
By using~\eqref{usa1} and~\eqref{usa2} in the expression~\eqref{crucial} of Lemma \ref{lemmasec} the statement follows.  \\
\end{proof}
\end{lemma}
With respect to the same splitting along $\gamma$, define the operator $\Tilde{\Omega}: \gamma^*TM \to \gamma^*TM$ as
\begin{equation}\label{omegatilde}
\tilde{\Omega}(V)=\Omega(V_1)+\Omega(V)_1+\frac{1}{2}\Omega(V_2)_2  .
\end{equation}
Consider the differential problem given by
\begin{equation}\label{parallel}
\dot{V}=\tilde{\Omega}(V)  .
\end{equation}
This is an ordinary linear system of first order differential equations which allows us to define a linear isomorphism $P_{\gamma}:T_{\gamma(0)}M \to  T_{\gamma(0)}M$ as follows  
\begin{equation*}
    P_{\gamma}(v)=V(T)  ,
\end{equation*}
where $V(t)$ is the unique solution of~\eqref{parallel} with initial vector $v$. 
\begin{lemma}\label{omegaortogonale}
The operator $\Tilde{\Omega}$ is antisymmetric with respect to the metric $g$. This in particular implies that the linear isomorphism $P_{\gamma}$ is orthogonal.
\begin{proof}
Let $V$ and $W$ be variations along $\gamma$. Then, by definition~\eqref{omegatilde} of $\tom$
\begin{align*}
 \langle \tom(V),W \rangle &= \langle \Omega(V_1)+\Omega(V)_1+\frac{1}{2}\Omega(V_2)_2 ,W \rangle && \\
&=-\langle V_1,\Omega(W)\rangle- \langle V,  \Omega(W_1) \rangle - \frac{1}{2} \langle V_2, \Omega(W_2) \rangle && \\
&=-\langle V, \Omega(W)_1 \rangle - \langle V, \Omega(W_1) \rangle - \frac{1}{2} \langle V,\Omega(W_2)_2 \rangle && \\ 
&=-\langle V, \tom(W)\rangle  .
\end{align*}
This implies that if $V$ and $W$ are solution of~\eqref{parallel}, then
\begin{equation*}
\frac{d}{\mathrm{d}t} \langle V ,W \rangle = \langle \tom(V),W \rangle + \langle V,\tom(W) \rangle=0  . 
\end{equation*}
By definition of $P_{\gamma}$, the statement holds. 

\end{proof}
\end{lemma}
Thus $P_{\gamma}$ is an orthogonal operator and in this sense the above construction generalizes the notion of Riemannian parallel transport to the magnetic case. We can now proceed to the proof of Theorem C which is an immediate consequence of the next statement.
\begin{lemma} \label{Klinglemma}
Let $(M,g,\sigma)$ be a magnetic system over an oriented even dimensional manifold and $\gamma$ a zero of $\eta_k$. If $\Msec>0$, then $\mathrm{index}(\gamma) \geq 1 $.
\begin{proof}
Assume $\Msec>0$. Consider $G \subset \gamma^*TM$ the orthogonal subbundle with respect to $\speed$ together with its projection map $\mathrm{p}_G: \gamma^*TM \to G$. Because $\tom(\speed)=\Omega(\speed)$, $\speed$ is a periodic solution of~\eqref{parallel} and the map $P_{\gamma}$ leaves invariant $G(0)$. Since $P_{\gamma}$ is orthogonal and since $M$ is oriented, we deduce that $\tilde{P}_{\gamma}=P_{\gamma} \circ \mathrm{p}_G$ is a special orthogonal isomorphism of $G(0)$. By assumption $G(0)$ is odd dimensional so that $1$ is eigenvalue of $\tilde{P}_{\gamma}$ i.e. there exists $v \in G$ such that $\tilde{P}_{\gamma}(v)=v$. By definition of $\tilde{P}_{\gamma}$, this is equal to say that there exists $V$ orthogonal to $\speed$ periodic solution of \eqref{parallel}. Let $V$ be such a solution and observe that $\tom(V)_2=\frac{1}{2}\Omega(V)_2$. Consider $W=V+g\speed$ and $\tau$ as in Lemma \ref{tecnica1}, and observe that 
\begin{align*}
\hess(W,\tau)&= \int_0^T  \Big|(\dot{V})_2-\frac{1}{2} \Omega(V)_2\Big|^2  \mathrm{d}t  -\int_0^T |V|^2\Msec\Big(\frac{\speed}{|\speed|},\frac{V}{|V|}\Big) \mathrm{d}t && \\
&=-\int_0^T |V|^2\Msec\Big(\frac{\speed}{|\speed|},\frac{V}{|V|} \Big) \mathrm{d}t<0 \ .
\end{align*}
Therefore, the index of $\gamma$ is at least one. 

\end{proof}
\end{lemma}
\begin{proof}[Proof of Theorem C] As argued for instance in \cite{Abb2}, if $k>c$, then $\eta_k$ carries a minimizer on each non trivial free homotopy class of loops of $M$. If $M$ is oriented and even dimensional and $\Msec>0$ for $k$ bigger than $c$, then by Lemma \ref{Klinglemma}, $\pi_1(M)$ has to be necessarily trivial. If $M$ is a closed oriented surfaces and the magnetic form is exact, by \cite{CMP}, $\eta_k$ carries a minimizer for every $k$ in the case $M$ is non simply connected and for every $k\in (0,c)$ otherwise. With the same argument we conclude that if $\Msec>0$, then $M=S^2$ and $k\geq c$.

\end{proof}
We end the section by generalizing a Bonnet-Myers theorem to our context. Preliminarily, we establish an explicit formula of $\hess$ when evaluated on vector field obtained by rescaling a solutions of the differential problem \eqref{parallel} with a $T$-periodic function with vanishing boundary conditions. 
\begin{lemma}\label{tecnica2}
Let $V$ be a solution of~\eqref{parallel} of unit norm and orthogonal to $\speed$. Consider $V^f=fV$, where $f:[0,T] \to \R$ with $f(0)=f(T)=0$. There exists a variation $W^f$ and a real number $\tau$ such that 
\begin{equation}
\hess (W^f,\tau)= \int_0^T \Big\{ \dot{f}^2 - f^2 \Msec \Big( \frac{\speed}{|\speed|} , V \Big) \Big\}  \mathrm{d}t \ .
\end{equation}
\begin{proof} Let $V^f$ as in the statement. Then by Lemma \ref{tecnica1}, there exist $g$ and $\tau$ such that, writing $W^f=V^f+g\speed$, it follows that 
\begin{align*}
\hess(W^f,\tau)&=\int_0^T \Big| (\dot{V}^f)_2 - \frac{1}{2} \Omega(V^f)_2  \Big|^2- |V^f|^2 \  \Msec \Big( \frac{\speed}{|\speed|}, V\Big) \mathrm{d}t && \\
&=\int_0^T \Big| \dot{f}V+f(\dot{V})_2 - f\frac{1}{2} \Omega(V)_2  \Big|^2- |fV|^2 \  \Msec \Big( \frac{\speed}{|\speed|}, V \Big) \mathrm{d}t && \\
&=\int_0^T \Big\{ \dot{f}^2-f^2\Msec \Big( \frac{\speed}{|\speed|}, V \Big) \Big\} \mathrm{d}t  .
\end{align*}
\end{proof}
\end{lemma}

\begin{lemma} \label{Bonnet-Myers} Let $\gamma=(x,T)$ be a zero of $\eta_k$ with $\mathrm{index}(\gamma)=m$. If $\Mric \geq \frac{1}{r^2}>0$ for a positive constant $r$, then 
\begin{equation*}
T\leq r \pi(m+1)
\end{equation*}

\begin{proof} Let $\{ V_1,...,V_{n-1} \}$ be a family of linearly independent solutions for the differential problem \eqref{parallel}, with unit norm and orthogonal to $\speed$. For $j=0,1,...,m$ and $i=1,...,n-1$ consider $V_i^{f_{j}}=f_{j}V_i $, where
\begin{equation*}
f_j(t)=
\left\{\begin{split}
\sin\Big( \frac{(m+1)\pi t}{T}\Big) \ \ \ , \ t\in \Big[\frac{jT}{m+1},\frac{(j+1)T}{m+1}\Big] \\
0\ \ \ \ \ \ \ \ \ \ \ \ \ \ \  \ \ \ \ \ \ \ \mathrm{otherwise}\ \ \ \ \  \ \ \ \ \ \ \ \ \ \ \ \    \\ 
\end{split}\right. 
\end{equation*}
 
Assume $\Mric \geq \frac{1}{r^2}>0$ and suppose by contradiction that $T>\pi r (m+1)$. Let $W_{i}^{f_j}=V_{i}^{f_j} + g_{ij}\speed$ and $\tau_{ij}$ given by Lemma \ref{tecnica2}. Fix $j$ and observe that the sum

\begin{align*}
\sum_{i=1}^{n-1} \hess (W_{i}^{f_j},\tau_{ij})  &=  \sum_{i=1}^{n-1} \int_{\frac{jT}{m+1}}^{\frac{(j+1)T}{m+1}} \Big\{ \frac{(m+1)^2\pi^2}{T^2} \cos^2 \Big( \frac{(m+1)\pi t}{T} \Big) \Big\} \mathrm{d}t && \\
& \ \ \ \ \ \ \ \ \ \ \ \ \ \ \ \ \ \ \ \ \ \ \ \ \ \ - \sum_{i=1}^{n-1} \int_{\frac{jT}{m+1}}^{\frac{(j+1)T}{m+1}} \Big\{\sin^2 \Big( \frac{(m+1)\pi t}{T}\Big) \Msec \Big(\frac{\speed}{|\speed|},V_i \Big) \Big\} \mathrm{d}t && \\ 
& =(n-1)\Big[ \frac{(m+1)^2\pi^2}{2T(m+1)}-\int_{\frac{jT}{m+1}}^{\frac{(j+1)T}{m+1}} \Big\{ \sin^2 \Big( \frac{(m+1)\pi t}{T} \Big) \Mric \Big( \frac{\speed}{|\speed|} \Big)  \Big\} \mathrm{d}t \Big] && \\
& \leq(n-1)\Big[ \frac{(m+1)^2\pi^2}{2T(m+1)}-\frac{T}{2 r^2 (m+1)}\Big] && \\
& =(n-1)\Big[ \frac{(m+1)^2\pi^2 r^2-T^2}{2T(m+1) r^2}\Big] <0  .
\end{align*}

Therefore, for every $j$ there exists an index $i_{j}\in \{ 0,1,...,n-1 \}$ such that $\hess(W_{i_j}^{f_j},\tau_{i_j j})<0$. By construction, for every $s\neq l$ the supports of $f_{s}$ and $f_{l}$ are disjoint which implies that $ \{ W_{i_j}^{f_j} \} $ is a family of linearly independent vector field along $\gamma$. For real coefficients $\lambda_1,...,\lambda_{m+1}$ define
\begin{equation*}
V=\sum_{j=1}^{m+1} \lambda_j V^{f_j}_{i_j} , \ g= \sum_{j=1}^{m+1} \lambda_j g_{i_j j} \ \mathrm{and} \ \tau= \sum_{j=1}^{m+1} \lambda_j \tau_{i_j j}.
\end{equation*}
By linearity of Equation \eqref{differentialproblem} and Equation \eqref{usa2}, $g$ and $\tau$ are exactly the ones associated to $V$ and $W=V+g\speed$ in Lemma \ref{tecnica1} in order to obtain that
\begin{align}
\hess \Big(\sum_{j=1}^{m+1}(W^{f_j},\tau_{i_j j})\Big)&= \int_0^T  \Big|(\dot{V})_2-\frac{1}{2} \Omega(V)_2\Big|^2  \mathrm{d}t  -\int_0^T |V|^2\Msec\Big(\frac{\speed}{|\speed|},\frac{V}{|V|}\Big) \mathrm{d}t \nonumber \\ 
&=\sum_{j=1}^{m+1} \lambda_j^2  \int_{\frac{jT}{m+1}}^{\frac{(j+1)T}{m+1}} \Big\{ \dot{f}_{ j}^2  - f_j^2 \  \Msec \Big(\frac{\speed}{|\speed|}, V_{i_j} \Big) \Big\} \mathrm{d}t \label{abc1}\\
&= \lambda_j^2 \sum_{j=1}^{m+1} \hess(W^{f_j},\tau_{i_j j})  \leq 0  , \label{abc2}
\end{align}
where in \eqref{abc1} we used again that the support of $f_j$ are disjoint. Observe that the equality in \eqref{abc2} holds if and only if $\lambda_j=0$ for every $j$. Summing up the family $\{ (W_{i_j}^j, \tau_{i_j j}) \}$ is a family of linear independent vectors of $T_{\gamma}\M$ which generate a vector subspace $Z$, of dimension $m+1$, such that $\hess|_{_Z}$ is negative definite. This implies that the index of $\gamma$ is equal or greater than $m+1$ in contradiction with the assumption. 

\end{proof}
\end{lemma}
Observe that the best constant one can obtain from Lemma \ref{Bonnet-Myers} is in fact optimal. This is confirmed for instance by the example the round 2-dimensional sphere with $\sigma =0 $ or the flat 2-dimensional torus with the volume form as magnetic form. Indeed, by fixing for simplicity $k=\frac{1}{2}$, an easy computation shows that in both case the Ricci magnetic curvature is constantly equal to $1$, all contractible closed magnetic geodesic with period $2\pi$ are zeroes of $\eta_k$ of index 1.

\section{Index estimates below $\man$ and proof of Theorem A (and Theorem A1)}\label{proofA}

\subsection{Weakly exact case} \label{weaklyexact}
\subsubsection{Minimax geometry below $\man$.}\label{minimaxdebole}
In this paragraph we assume $\sigma$ to be weakly exact. Let $\prim$ be the primitive of $\eta_k$ defined on $\M_0$ as given in in~\eqref{primcon}. Below the critical value, $\prim$ enjoys a minimax geometry on the set of loops with short length which we discuss now. Denote by $M^+=M\times (0,+\infty)$ the set of constant loops with free period and consider 

\begin{equation*}
    \Gamma^k_0=\{ \phi:[0,1]\to \M_0, \ \phi(0) \in M^+ \ \mathrm{and} \ S_k^{\sigma}(\phi(1))<0 \}  .
\end{equation*}
By definition \eqref{mane} of $c$, when $k<c$ the set $\Gamma^k_0$ is non empty and the minimax value function is defined by
 
 \begin{equation}
       \mathbf{s}:(0,c) \to \R   , \ \mathbf{s}(k)=\inf_{\Gamma^k_0}\max_{[0,1]} \prim(\phi(t))  . 
    \end{equation}
     
\begin{lemma}\label{minimaxlemma}
    Let $I$ be an open interval with compact closure fully contained in $(0,c)$. Then there exists a positive $\varepsilon=\varepsilon(I)$ such that
    \begin{equation}
        \mathbf{s}(k)\geq \varepsilon   , \ \forall k\in I .
    \end{equation}
    \begin{proof}
    For the proof we refer the reader to \cite[Section~5]{AB}.
    
        \end{proof}
\end{lemma}
Consider on $\M_0$ the negative gradient vector field $-\nabla \prim$, induced by the Riemannian metric $g_{\M}$. Since $\prim$ is smooth, $-\nabla \prim$ admits a local flow for positive time which in general is not complete. In fact, source of non completeness mainly originates from the non completeness of $(0,+\infty)$ and the fact that $|\nabla \prim|_{\M}$ is not bounded on $\M_0$. One can avoid these situations by considering the pseudo gradient given by
\begin{equation}\label{pseudow}
    \pgr = (h \circ \prim) \frac{\nabla \prim}{\sqrt{1+|\nabla \prim|_{\M}^2}}  ,
\end{equation}
where $h:\R \to [0,1]$ is a cut-off function such that $h^{-1}\{ 0 \}=(-\infty , \frac{s(k)}{4}]$ and $h^{-1}\{ 1 \}=[\frac{s(k)}{2}, +\infty)$. As argued in \cite[Section~8]{Abb2}, because $\pgr$ has bounded norm and vanishes on the set $\{ \prim < \frac{s(k)}{4} \}$, it admits a positively complete flow $\flgr: \M_0 \times [0,+\infty) \to \M_0$. The next lemma evidences two important properties of $\flgr$.

\begin{lemma}\label{gradientlemma}
    Let $u=(x,T): \rp \to \M_0$ be an integral curve of $\flgr$. Then for every $s \in \rp$ the following statements hold:
   \begin{itemize}
   \item [(i)] $\prim(u(s)) \leq \prim(u(0))$,
   \item [(ii)] $|T(s)-T(0)|^2 \leq s\Big( \prim (u(0)) - \prim (u(s))\Big)$.
   \end{itemize}
   \begin{proof} 
   For the proof we refer the reader to \cite[Section~5]{Merry}. 
   
   \end{proof}
\end{lemma}

\begin{lemma}
The set $\Gamma_0^k$ is $\flgr$-invariant. 
\end{lemma}
\begin{proof}
Since for every $(p,T) \in M^+$, $\nabla \prim (p,T) = k\frac{\partial}{\partial T}$ then $\flgr( t  , M^+) \subseteq M^+$ for every positive time $t$. This fact together with point (i) of Lemma \ref{gradientlemma} implies that $\Gamma_0^k$ is $\flgr$-invariant. 

\end{proof}
\subsubsection{Existence almost everywhere and index estimates.}
Hereafter we fix an interval $I$ with compact closure fully contained in $(0,c)$. Because $\prim$ is monotone increasing with respect to the parameter $k$, one can easily deduce that also $\mathbf{s}$ is monotone increasing. Thus, there exists $J\subset I$ of full Lebesgue measure such that $\mathbf{s}$ is differentiable on $J$. Let us point out that the almost every where differentiability of $\mathbf{s}$ plays a central role in the Struwe monotonicity argument. Indeed, this technic allows to recover converging Palais-Smale sequence related to the minimax geometry pointed out in Section \ref{minimaxdebole} just for $k\in J$ \cite[Section~8]{Abb2},\cite[Section~5]{Merry}. In general no information can be deducted when $k$ is not a point of differentiability. Here we retrace a more precise version of this statement. Our goal is to prove in Lemma \ref{index1}, that vanishing points of $\eta_k$ coming from this minimax geometry have Morse index bounded by one. Similar constructions were previously considered in \cite{AMP} and \cite{AAMP}. We proceed by showing that at the values of the energy where $\mathbf{s}$ is differentiable we can bound the period of points in the image of elements of $\Gamma_0^k$ which almost realize the minimax value. 
\begin{lemma}\label{periodolimitato}
Let $k\in J$. Then there exists a constant $A$ such that for every small $\varepsilon>0$, by writing $k_{\varepsilon}=k+\varepsilon$, if $\phi \in \Gamma_0^{k}$ and $\max S^{\sigma}_{k_\varepsilon}(\phi) \leq \mathbf{s}(k_{\varepsilon}) + \varepsilon$, then for every $t\in[0,1]$ satisfying also $\prim(\phi(t))> \mathbf{s}(k)- \varepsilon$, the inequality $T_{\phi(t)} \leq A+2$ holds. 
\begin{proof}
By the differentiability of $\mathbf{s}$ at $k$, there exists a positive constant $A=A(k)$ such that for every positive $\varepsilon$,
    \begin{equation}\label{costanteA}
    |\mathbf{s}(k_{\varepsilon})-\mathbf{s}(k)|\leq A \cdot |k_{\varepsilon}-k|  .
    \end{equation} 
Let $\phi \in \Gamma_0^k$ and $t \in [0,1]$ be as in the statement. It holds that
 \begin{equation*}
     T_{\phi(t)}=\frac{S^{\sigma}_{k_\varepsilon} (\phi(t))-\prim (\phi(t))}{k_{\varepsilon}-k}\leq \frac{\mathbf{s}(k_\varepsilon)+\varepsilon-\mathbf{s}(k)+\varepsilon}{\varepsilon} \leq  A+2  .
 \end{equation*}

\end{proof}
\end{lemma}
Let $T^*$ such that  
\begin{equation}\label{tstar}
T^*> \sqrt{2\mathbf{s}(k)}+3(A+1),
\end{equation}
for which the role is clarified later. Consider the set 
\begin{equation*}
\B_k= \{ \prim \geq \frac{\mathbf{s}(k)}{2} \} \cap \{ T < T^* \} .
\end{equation*}
Observe that, by Lemma \ref{minimaxlemma}, $\mathbf{s}(k)$ is positive and as consequence of \cite[Proposition~5.8]{Merry}, $\prim$ restricted to $\B_k$ satisfies Palais-Smale. In particular, the set $\C_k = \mathrm{Crit}(\prim) \cap \B_k = \Z(\eta_k) \cap \B_k$ is compact. The next statement shows that $\C_k$ is non empty and that we can always find an element of $\Gamma_0^k$ which passes arbitrarily close to $\C_k$. 

\begin{lemma}\label{struwe1}
    Let $k\in J$. For every $V$ open neighborhood of $\C_k$ and for every small $\varepsilon>0$ there exists an element $\varphi_{\varepsilon} \in \Gamma_0^k$ such that
    \begin{equation*}
        \varphi_{\varepsilon} ([0,1])\subset \{ \prim < \mathbf{s}(k) \} \cup \Big(\{ \mathbf{s}(k)\leq \prim < \mathbf{s}(k)+\varepsilon \} \cap V \Big)  .
    \end{equation*}
    In particular, $\C_k$ is non empty.
    \begin{proof} Let $k$ be a point of differentiability for $\mathbf{s}$ and let $V$ be an open neighborhood of $\C_k$. Observe that $\C_k$ consists of fixed points for $\flgr$ and $\prim$ satisfies Palais-Smale on $\B_k$ so that there exists an open set $V^{'}$ and a positive $\delta=\delta(V^{'})$ such that $V^{'}$ still contains $\C_k$, $\flgr([0,1] \times V^{'}) \subset V \cap \B_k$ and
  \begin{equation}\label{boundgradient}
            |\nabla \prim|_{\M} \geq \delta >0 \ \  \mathrm{on} \ \ \forall (x,T) \in \B_k \setminus V^{'}  . 
        \end{equation}
Let $\varepsilon$ be such that $\varepsilon < \delta^2$. By definition of $\mathbf{s}$, for every $\tilde{\varepsilon} \in (0, \varepsilon)$ there exists $ \phi_{\tilde{\varepsilon}} \in \Gamma_0^{k_{\tilde{\varepsilon}}} $ such that 
\begin{equation*}
\max_{t \in [0,1]} S^{\sigma}_{k_{\tilde{\varepsilon}}}( \phi_{\tilde{\varepsilon}}(t)) \leq \mathbf{s}(k_{\tilde{\varepsilon}}) + \tilde{\varepsilon} . 
\end{equation*}
Up to shrinking $\tilde{\varepsilon}$, we can assume that $\tilde{\varepsilon}(A+1)<\varepsilon$ and that $\phi_{\tilde{\varepsilon}} \in \Gamma_0^k$. By \eqref{costanteA} in Lemma \ref{periodolimitato}, we deduce that
\begin{equation*}
 \prim(\phi_{\tilde{\varepsilon}}(t)) < S^{\sigma}_{k_{\tilde{\varepsilon}}}( \phi_{\tilde{\varepsilon}}(t)) \leq \mathbf{s}(k_{\tilde{\varepsilon}}) + \tilde{\varepsilon} - \mathbf{s}(k) + \mathbf{s}(k) \leq \mathbf{s}(k)+(A+1)\tilde{\varepsilon}<\mathbf{s}(k) + \varepsilon \ .
\end{equation*}
We claim that the element $\varphi_{\varepsilon} \in \Gamma_0^k$ defined by
\begin{equation}\label{element}
\varphi_{\varepsilon} = \flgr(1,\phi_{\tilde{\varepsilon}}(t)) , 
\end{equation}
is the desired one. In fact, because $\prim$ decreases along the flow lines of $\flgr$, if $t \in [0,1]$ is such that $\prim(\phi_{\tilde{\varepsilon}}(t))< s(k)$ then $\prim(\varphi_{\varepsilon} (t))< s(k)$. By the choice of $T^*$ in \eqref{tstar}, by Lemma \ref{periodolimitato} and by Lemma \ref{gradientlemma}, if $\prim (\phi_{\tilde{\varepsilon}} (t)) \in (\mathbf{s}(k), \mathbf{s}(k) + \varepsilon)$, then either $\prim(\varphi_{\varepsilon} (t))< s(k)$ or $\flgr(s, \phi_{\tilde{\varepsilon}}(t) ) \in\B_k \cap (\mathbf{s}(k), \mathbf{s}(k) + \varepsilon) $ for every $s \in [0,1]$. Let us focus on the second case. If there exists a time $s_0\in [0,1]$ such that $\flgr (s_0,  \phi_{\tilde{\varepsilon}}(t)) \in V^{'}$ then, by the assumptions on $V^{'}$, we deduce that $\flgr ( 1 , \phi_{\tilde{\varepsilon}}(t)) \in V^{'}$. Suppose by contradiction that for every $t$ such that $\flgr(s, \phi_{\tilde{\varepsilon}}(t) ) \in\B_k \cap (\mathbf{s}(k), \mathbf{s}(k) + \varepsilon) $ and for every $s \in [0,1]$, $\flgr(s, \phi_{\tilde{\varepsilon}}(t) )$ does not enter $V^{'}$. Then, by \eqref{boundgradient}, it follows that
\begin{equation*} \label{eq1}
\begin{split}
    \prim(\varphi_{\varepsilon} (t)) & = \prim (\phi_{\tilde{\varepsilon}} (t)) + \int_0^1 \frac{d}{ds}\prim(\flgr(s,\phi_{\tilde{\varepsilon}} (t))) \mathrm{d}s  \\
    & \leq \mathbf{s}(k) + \varepsilon - \int_0^1 |\nabla \prim |_{_{\M}}^2 \mathrm{d}s \\
    & \leq \mathbf{s}(k) +\varepsilon - \delta^2 \\
    & < \mathbf{s}(k) \ .
\end{split}
    \end{equation*}
In this way we find an element of $\Gamma_0^k$ such that $ \varphi_{\varepsilon}([0,1]) \subset \{  \prim < \mathbf{s}(k) \} $ which contradicts the definition of $\mathbf{s}(k)$. We conclude that $\varphi_{\varepsilon}$, as defined in~\eqref{element}, is such that either $\prim(\varphi_{\varepsilon}(t))<\mathbf{s}(k)$ or $\prim(\varphi_{\varepsilon}(t))\in (\mathbf{s}(k),\mathbf{s}(k) + \varepsilon)$ and $\varphi_{\varepsilon}(t) \in V^{'}\subset V$. The claim is proved.

    \end{proof}
\end{lemma}

\begin{lemma}\label{index1} Let $k\in J$. There exists $\gamma \in \Z(\eta_k)$ such that $\mathrm{index}(\gamma)\leq 1$.
\begin{proof}
By Lemma \ref{struwe1}, the set $\C_k$ is not empty. Let $\gamma \in \C_k$ and consider a local chart of coordinates $(\V_{\gamma},\mathrm{\Phi}_{\gamma})$ and $D_{\gamma}$ chosen as in Lemma \ref{chart}. Let $\rho^t_{\gamma}: B_{r_{\gamma}}^H \times (B_{r_{\gamma}}^E \setminus \{ 0 \} ) \to B_{r_{\gamma}}^H \times B^E_{r_{\gamma}}$ be the deformation given by 
\begin{equation*}
\rho^t_{\gamma}(y_h,y_e )=\Big(y_h,y_e + \min \{  t, r_{\gamma}-|y_e|_0  \} \frac{y_e}{|y_e|_0} \Big)  . 
\end{equation*}
Observe that $\rho^t_{\gamma}$ pushes points of $B_{r_{\gamma}}^H \times (B_{r_{\gamma}}^E \setminus \{ 0 \} )$ towards the boundary, in the direction of $E$. In particular, if $t<r_{\gamma}-|y_e|_0$, then by point (ii) of Lemma \ref{chart}, it holds that
\begin{equation}\label{chart2}
\locprim(\rho^t_{\gamma}(y_h,y_e)) \leq \locprim(y_h,y_e) - D_{\gamma} t^2 .
\end{equation}
Denote by $\W_{\gamma}=\mathrm{\Phi}_{\gamma}^{-1}(B_{\sfrac{r_{\gamma}}{4}}^H \times B^E_{\sfrac{r_{\gamma}}{4}})$ and by $\U_{\gamma}=\mathrm{\Phi}_{\gamma}^{-1}(B_{\sfrac{r_{\gamma}}{2}}^H \times B^E_{\sfrac{r_{\gamma}}{2}})$. Let $\chi_{\gamma}$ be a bump function on $\M_0$, supported in $\V_{\gamma}$ and such that $\chi_{\gamma} (\U_{\gamma})=1$. By contradiction suppose that every $\gamma \in \C_k$ has index greater than 1. Since $\C_k$ is compact, there exist $\gamma_1,..., \gamma_n\in \C_k$ and local charts $(\V_{\gamma_i},\mathrm{\Phi}_{\gamma_i})$ such that $\C_k \subset \bigcup_{i=1}^n \W_{\gamma_i}$. Let us remark that the change of coordinates is bi-Lipschitz because is composition of bi-Lipschitz map. We write $C_{ij}$ the Lipschitz constant of the change of coordinates $\mathrm{\Phi}_{\gamma_j}\circ  \mathrm{\Phi}_{\gamma_i}^{-1}$ and by $|  \cdot  |_{0,i}$ the norm on $\mathrm{\Phi}_{\gamma_i}(\V_{\gamma_i})$ induced by the chart $\mathrm{\Phi}_{\gamma_i}$. \\
Let $\delta=\min_i D_{\gamma_i} $,  $R=\min_i r_{\gamma_i} $, $  C=\max_{i\neq j} C_{ij} $ and fix 
\begin{equation*}
\varepsilon \in \Big(0,\min \Big\{ \frac{R}{4nC}, \frac{R}{4n} \Big\} \Big)  .
\end{equation*}
For $j=1,...,n$ consider the set 
\begin{equation*}
\W_{\gamma_i, j} =  \mathrm{\Phi}_{\gamma_i}^{-1} \Big(\bigcup_{z \in \mathrm{\Phi}_{\gamma_i}(\W_{\gamma_i})}B(z,j\varepsilon C)\Big)  , 
\end{equation*}
where $B(z,j\varepsilon C) \subset \mathrm{\Phi}_{\gamma_i}(\W_{\gamma_i})$ is the ball centered at $z$ with radius $j\varepsilon C$ with respect to the $| \cdot |_{0,i}$. The following inclusions hold
\begin{equation*}
\W_{\gamma_i} \subset \W_{\gamma_i, 1} \subset \W_{\gamma_i, 2} \subset ... \subset \W_{\gamma_i, j}\subset \W_{\gamma_i, j+1} \subset ... \subset \W_{\gamma_i, n} \subset \U_{\gamma_i} \subset \V_{\gamma_i}  . 
\end{equation*}
By Lemma \ref{struwe1}, there exists an element $\phi \in \Gamma_0^k$ such that 
\begin{equation}\label{inclusion}
\phi([0,1]) \subset \{ \prim < \mathbf{s}(k) \} \cup \Big( \{ \mathbf{s}(k) \leq \prim <\mathbf{s}(k) + \varepsilon^2 \delta \} \cap \bigcup_{i=1}^n \W_{\gamma_i} \Big)  . 
\end{equation}
Up to refining the cover, we can assume that the endpoints $\phi(0),\phi(1) \notin \bigcup_{i=1}^n \V_{\gamma_i}$. By point (i) of Lemma \ref{chart}, $\mathrm{\Phi}_{\gamma_i} (\C_k \cap \mathrm{\Phi}_{\gamma_i}^{-1}(\V_{\gamma_i})) \subseteq B^H_{r_{\gamma_i}} \times \{ 0 \} $ for every $i$, and by contradiction we are assuming that $B^H_{r_{\gamma_i}} \times \{ 0 \} $ has codimension bigger than one. Since the domain of $\phi$ has dimension one, with the help of the transversality theorem \cite[Section~5]{trans}, we can find an element $\phi_0$ arbitrarily $C^0$ close to $\phi$ such that its image $\phi_0([0,1]) $ does not intersect $B^H_{r_{\gamma}} \times \{ 0 \}$ and it is still contained in the right-hand term of the inclusion~\eqref{inclusion}. \\
Define 
\begin{equation}\label{deformazione}
\phi_1 (t)=
\left\{\begin{split}
\mathrm{\Phi}_{\gamma_1}^{-1} \Big( \rho_{_{\gamma_1}}^{\varepsilon \chi_{_{\gamma_1}} (\phi_0(t))} (\mathrm{\Phi}_{\gamma_1}(\phi_0(t)))\Big) \ \  \mathrm{if} \ \phi_0(t)\in \V_{\gamma_1} \ ,  \\
\phi_{0}(t) \ \ \ \ \ \ \ \ \ \ \ \ \ \ \mathrm{otherwise}\ .  \ \ \ \  \ \ \ \ \ \ \ \ \ \ \ \    \\ 
\end{split}\right.
\end{equation}
Observe that $\phi_1 \in \Gamma_0^k$ because it is obtained by deforming continuously the segment of $\phi_0$ contained in $V_{\gamma_1}$ by keeping the endpoints fixed. Moreover, if $\phi_0(t) \in \{ \mathbf{s}(k) \leq \prim < \mathbf{s}(k) + \varepsilon^2 \delta \} \cap \W_{\gamma_1}$, then by the choice of $\varepsilon$ and by \eqref{chart2} if follows that 
\begin{align*}
\prim( \phi_1(t))& \leq \prim( \phi_0(t)) - D_{\gamma_1} \varepsilon^2 &&\\
& \leq \mathbf{s}(k) + \delta \varepsilon^2 - D_{\gamma_1}\varepsilon^2 &&\\
&< \mathbf{s}(k)  . 
\end{align*}
Finally, because $ \rho_{_{\gamma_1}}^{\varepsilon \chi_{_{\gamma_1}}} $ pushes points of $\mathrm{\Phi}_{\gamma_1}(\W_{\gamma_1}) $ at distance at most $\varepsilon$, if $\phi_0(t) \in \W_{\gamma_1} \cap \W_{\gamma_i} $ then
\begin{equation}\label{lipschitz}
\Big|     \mathrm{\Phi}_{\gamma_i}  (\phi_1(t)) -  \mathrm{\Phi}_{\gamma_i} (\phi_0(t))   \Big|_{0,i}  \leq C_{1i} \Big|     \mathrm{\Phi}_{\gamma_1} (\phi_1(t)) -   \mathrm{\Phi}_{\gamma_1} (\phi_0(t))   \Big|_{0,1} \leq C\varepsilon   , 
\end{equation}
so that $\phi_1(t) \in \W_{\gamma_i,1}$. \\
Summing up the element $\phi_1$ enjoys 
\begin{equation*}
\phi_1([0,1]) \subset \{ \prim < \mathbf{s}(k) \} \cup \Big( \{ \mathbf{s}(k) \leq \prim < \mathbf{s}(k)+\varepsilon^2 \delta \} \cap \bigcup_{i=2}^n \W_{\gamma_i,1} \Big)  . 
\end{equation*}
We repeat the construction as follows. Define $\phi_2 \in \Gamma_0^k$ as 
\begin{equation*}
\phi_2 (t)=
\left\{\begin{split}
\mathrm{\Phi}_{\gamma_2}^{-1} \Big( \rho_{_{\gamma_2}}^{\varepsilon \chi_{_{\gamma_2}} (\phi_1(t))} (\mathrm{\Phi}_{\gamma_2}(\phi_1(t)))\Big) \ \  \mathrm{if} \ \phi_1(t)\in \V_{\gamma_1} \ ,  \\
\phi_{1}(t) \ \ \ \ \ \ \ \ \ \ \ \ \ \ \mathrm{otherwise}\ .  \ \ \ \  \ \ \ \ \ \ \ \ \ \ \ \    \\ 
\end{split}\right.
\end{equation*}
By using again the inequality~\eqref{chart2}, if $\phi_1(t) \in \W_{\gamma_2,1}$ then 
\begin{align*}
\prim( \phi_2(t))& \leq \prim( \phi_1(t)) - D_{\gamma_2} \varepsilon^2 &&\\
& \leq \mathbf{s}(k) + \delta \varepsilon^2 - D_{\gamma_2}\varepsilon^2 &&\\
&< \mathbf{s}(k)   .
\end{align*}
Moreover, if $\phi_1(t) \in \W_{\gamma_1,1} \cap \W_{\gamma_i,1}$ for some $i\in \{ 2,...,n \}$, then
\begin{align*}
\Big|     \mathrm{\Phi}_{\gamma_i}  (\phi_2(t)) -  \mathrm{\Phi}_{\gamma_i}  (\phi_1(t))   \Big|_{0,i} & \leq C_{2i} \Big|     \mathrm{\Phi}_{\gamma_2} (\phi_2(t)) -   \mathrm{\Phi}_{\gamma_2}(\phi_1(t))   \Big|_{0,2} \leq C\varepsilon   ,  
\end{align*}
which implies that $\phi_2(t) \in \W_{\gamma_i,2}$. Thus, the element $\phi_2$ satisfies
 \begin{equation*}
\phi_2([0,1]) \subset \{ \prim < \mathbf{s}(k) \} \cup \Big( \{ \mathbf{s}(k) \leq \prim < \mathbf{s}(k)+\varepsilon^2 \delta \} \cap \bigcup_{i=3}^n \W_{\gamma_i,2} \Big)  . 
\end{equation*}
Iterating the process, at the step $m$ we find $\phi_m$ such that 
 \begin{equation*}
\phi_m([0,1]) \subset \{ \prim < \mathbf{s}(k) \} \cup \Big( \{ \mathbf{s}(k) \leq \prim <\mathbf{s}(k) +  \varepsilon^2 \delta \} \cap \bigcup_{i=m+1}^n \W_{\gamma_i,m} \Big)  ,
\end{equation*}
at the step $n$, we find an element $\phi_n$ such that 
 \begin{equation*}
\phi_n([0,1]) \subset \{ \prim < \mathbf{s}(k) \}  ,
\end{equation*}
in contradiction with the definition of $\mathbf{s}(k)$.

\end{proof}
\end{lemma}

\subsection{Non weakly exact case}\label{nonweaklyexact}
\subsubsection{Variation of $\eta_k$ along paths}
Assume $\sigma$ is not weakly exact. In particular, $\man= + \infty$ and $\eta_k$ does not admit a primitive globally defined on $\M_0$. Despite that, if $u:[0,1] \to \M_0$ is a continuous path the action variation $\av (u):[0,1] \to \R$ of $\eta_k$ along $u$ is always well defined and it is given by 
\begin{equation*}
\av(u)(t) = \int_0^t u^* \eta_k  . 
\end{equation*}
As argued in \cite[Section~2.3]{AB}, if $\U \subseteq \M_0$ is an open subset and $u([0,1]) \subset \U$, then for every local primitive $\prim: \U \to \R$ we have
\begin{equation}\label{primug}
\av(u)(t)=\prim(u(t)) - \prim(u(0))   , \ \forall t \in [0,1]  . 
\end{equation}
Let $u:[0,1]\times[0,R] \to \M_0$ be a homotopy of $u$ with the starting point fixed and write $u_s=u(\cdot,s)$ and by $u^t=u(t, \cdot) $. The fact that $\eta_k$ is closed implies that
\begin{equation}\label{omotopia}
\av(u_s)(t)=\av(u_0)(t) + \av(u^t)(R)  .
\end{equation}
\subsubsection{Minimax geometry}
Consider $M^+$ and $\V_{\delta}$ as in the paragraph \ref{minimaxdebole}. If $\delta$ is small enough, then we can restrict definition~\eqref{primcon} of $\prim$ to set of short loops $\V_{\delta}\subset \M_0$. Moreover, as argued in \cite[Lemma~3.2]{AB}, such a primitive also enjoys that
\begin{equation*}
\inf_{\V_{\delta}} \prim = 0  .
\end{equation*}
By identifying the north and the south pole of the unit sphere $S^2$ with the end points of the interval $[0,1]$, we have a 1-1 correspondence between
\begin{equation*}
\Big\{ f:S^2 \to M \Big\} \xlongrightarrow {\text{F}} \Big\{ \phi: \Big( [0,1],\{ 0, 1 \} \Big) \to \Big(\M_0, M^+   \Big) \Big\}  .
\end{equation*}
It is a well known fact that $F$ descends to the homotopy quotient \cite{Kli1}. Because $\sigma$ is non weakly exact $\pi_2(M) \neq \{ 0 \}$ so that given a non trivial element $\mathfrak{u} \in \pi_2(M)$ one can consider the following set 
\begin{equation*}
\Gamma_{\mathfrak{u}} = \Big\{ \phi: \Big( [0,1],\{ 0, 1 \} \Big) \to \Big(\M_0, M^+   \Big), \ F^{-1}(\phi) \in \mathfrak{u} \Big\} . 
\end{equation*}
For $\phi \in \Gamma_{\mathfrak{u}}$ define a primitive $\prim (\phi) :[0,1] \to \R $ of $\eta_k$ along $\phi$ by 
\begin{equation}\label{primele}
\prim (\phi)(t) = \av(\phi)(t) + T_{\phi(0)}k . 
\end{equation}
In this setting, the minimax value function is given by 
\begin{equation*}
\mathbf{s}^{\mathfrak{u}}: \rp \to \rp   , \  \mathbf{s}^{\mathfrak{u}}(k) = \inf_{\phi \in \Gamma^{\mathfrak{u}}} \max_{t\in[0,1]} \pact (\phi)(t)  . 
\end{equation*}
In analogy with the weakly exact case, the following lemma holds.
\begin{lemma}\label{monotone} Let $I$ be an open interval with compact closure fully contained in $\rp$. There exists a positive $\varepsilon$ such that 
\begin{equation}
\su(k) \geq \varepsilon  , \ \forall t \in I  . 
\end{equation}
Moreover $\su$ is monotone increasing on $I$.
\begin{proof}
The proof is contained in \cite[Section~4]{AB}. 

\end{proof}
\end{lemma}
Let $N_k = \{ (x,T) \in \V_{\delta} \ | \ \prim(x,T) < \frac{\varepsilon}{4} \}$ and let $h$ be a cut-off function such that  $h^{-1}(1)=[\frac{\varepsilon}{2}, +\infty )$ and $h^{-1}(0)=(-\infty , \frac{\varepsilon}{4}]$. Define $\tilde{h}: \M_0 \to \R$ as follows
\begin{equation*}
\tilde{h} (x,T)=
\left\{\begin{split}
 \ \ 1 \ \  \mathrm{if} \ (x,T) \in \M_0\setminus N_k   \\
h \circ \prim  \ \ \mathrm{if} \  (x,T) \in N_k    \ \ 
\end{split}\right.
 \end{equation*}
Let $X_k$ be the vector field on $\M_0$ given by 
\begin{equation*}
X_k = \tilde{h} \cdot \frac{ \eta^{\#}_k}{\sqrt{1+ | \eta^{\#}_k|^2_{\M}}}  , 
\end{equation*}
where $ \eta^{\#}_k$ is the dual vector field of the 1-form $\eta_k$ obtained through the Riemannian metric $g_{\M}$. In accordance with the weakly exact case, the pseudo gradient $\pgr$ enjoys the following properties
\begin{lemma}\label{flussograd} The flow $\flgr$ of $\pgr$ is positively complete and $\Gamma_{\mathfrak{u}}$ is a $\flgr$-invariant set. Moreover, if $u=(x,T):\rp \to \M_0$ is a flow line of $\flgr$, then for every $s \in \rp$, 
\begin{itemize}
\item [(i)] $\av (u) (s) \leq 0$, 
\item[(ii)] $|T(s)-T(0)|^2 \leq -s \av(u)(s)$. 
\end{itemize}

\begin{proof}
The details of the proof are in \cite[Proposition~2.8~and~Lemma~2.9]{AB}.

\end{proof}
\end{lemma}

\subsubsection{Existence almost everywhere and index estimates}
In what follows we readapt Lemma \ref{periodolimitato}, Lemma \ref{struwe1} and Lemma \ref{index1} to the minimax geometry of $\eta_k$ on $\Gamma^{\mathfrak{u}}$. By Lemma \ref{monotone}, $\su$ is monotone increasing. Therefore there exists $J\subset I$ of full Lebesgue measure such that $\su$ is differentiable on $J$. 
\begin{lemma}\label{periodo}
Let $k \in J$. Then there exists a constant $A>0$ such that $\forall \varepsilon >0$, by writing $k_{\varepsilon}= k + \varepsilon$, if $\phi \in \suk $ is such that $\prim (\phi)(t) \leq \su(k_{\varepsilon}) + \varepsilon$ for all $t \in [0,1]$, then when $\phi(t)$ satisfies also $\prim (\phi)(t) > \su(k) - \varepsilon$, the inequality $T_{\phi(t)}\leq A$ holds.
\begin{proof}
The fact that $k$ is a point of differentiability for $\su$ implies the existence of a positive constant $A=A(k)$ such that 
\begin{equation*}
|\su ( k_{\varepsilon}) - \su(k) | \leq A \cdot |k_{\varepsilon}-k| \ , \ \forall k_{\varepsilon} \in I. 
\end{equation*}
Fix $\varepsilon>0$ and $\phi \in \suk$. If $\phi(t)$ is as in the statement then 
\begin{equation*}
T_{\phi(t)} = \frac{S^{\sigma}_{k_{\varepsilon}}(\phi)(t) - \prim(\phi)(t)}{k_{\varepsilon} - k} \leq \frac{\su(k_{\varepsilon})+ \varepsilon - \su(k) + \varepsilon}{\varepsilon}\leq A+2.
\end{equation*}

\end{proof}
\end{lemma}
Let again T such that 
\begin{equation}\label{tstarn}
T^* = \sqrt{2\su(k)}+3(A+1),
\end{equation}
and consider the set
\begin{equation*}
\B_k =( \M_0 \setminus N_k) \cap \{ T\leq T^* \}  .
\end{equation*}
By \cite[Proposition~5.8]{Merry} and by \cite[Theorem~2.6]{AB}, every vanishing sequence of $\eta_k$ contained in $\B_k$ is compact so that also the set $\C_k = \B_k \cap \Z(\eta_k)$ is compact. 
\begin{lemma}\label{struwe2} Let $k \in J$. Then for every $V$ open neighborhood of $\C_k$ and for every $\varepsilon>0$ there exists an element $\varphi_{\varepsilon} \in \suk$ such that, $\forall t \in [0,1]$ either $\prim(\varphi_{\varepsilon})(t) < s^{\mathrm{u}}(k)$ or $\prim(\varphi_{\varepsilon})(t) \in [s^{\mathrm{u}}(k) ,  \su +\varepsilon)$ and $\varphi_{\varepsilon}(t) \in V$. In particular,  $\C_k$ is non empty.

\begin{proof}
Because $\C_k$ consists of zeros for $\eta_k$ and vanishing sequences of $\B_k$ converge, there exists $V^{'} \subset V$ and a positive $\delta=\delta(V^{'})$ such that $V^{'}$ still contains $\C_k$, $\flgr ([0,1] \times V^{'}) \subset V \cap \B_k$ and 
\begin{equation}\label{d1}
| \eta^{\#}_k |_{\M_0} \geq \delta   , \  \forall (x,T) \in \B_k \setminus V^{'}.
\end{equation}
Let $\varepsilon$ be such that $\varepsilon< \delta^2$. By definition of $\su$, for every $ \tilde{\varepsilon} \in (0, \varepsilon)$ there exists $\phi_{\tilde{\varepsilon}} \in \suk$ such that 
\begin{equation*}
\max_{t \in [0,1]} \prim(\phi_{\tilde{\tilde{\varepsilon}}})(t) < \su(k_{\tilde{\varepsilon}}) + \tilde{\varepsilon}  .
\end{equation*}
For $s\in[0,1]$ define 
\begin{equation}\label{d2}
 \varphi^s_{\varepsilon}(t)= \flgr (s, \varphi_{\tilde{\varepsilon}}(t)) \ , 
\end{equation}
and observe that, by~\eqref{omotopia},
\begin{equation}\label{d3}
\prim(\varphi^s_{\varepsilon})(t)= \prim(\phi_{\tilde{\varepsilon}})(t) + \av \Big(\flgr ( \cdot, \varphi_{\tilde{\varepsilon}}(t))\Big)(s) .
\end{equation}
Similarly to the weakly exact case, we claim that the element $\varphi_{\varepsilon}=\varphi^1_{\varepsilon} \in \suk$ is the desired one. Indeed, if $t \in [0,1]$ is such that $\prim (\phi_{\tilde{\varepsilon}})(t) < \su(k)$, by point (i) of Lemma \ref{flussograd}, $\av$ is non positive along flow lines of $\flgr$ and \eqref{d3} consequently implies that $\prim( \varphi_{\varepsilon})(t)< \su(k)$. On the other hand, if $\prim(\phi_{\tilde{\varepsilon}})(t) \in [\su(k), \su(k) + \varepsilon)$, by the choice of $T^*$ in \eqref{tstarn}, by Lemma \ref{periodo} and Lemma \ref{flussograd}, it follows that either $\prim( \varphi_{\varepsilon})(t)< \su(k)$ or $\prim( \varphi_{\varepsilon})(t)\in[\su(k), \su(k) + \varepsilon)$ and $\varphi^{s}_{\varepsilon}(t) \in \B_k$ for every $s \in [0,1]$. Let us focus on the second case. By the assumptions on $V^{'}$, if for some time $s_0 \in [0,1]$ it follows that $\varphi^{s_0}_{\varepsilon}(t)\in V^{'}$ so that also $\varphi_{\varepsilon}(t)$ is still belonging to $V^{'}$. Suppose by contradiction that $\varphi^s_{\varepsilon}(t)$ does not enter $V^{'}$ for every $s\in [0,1]$.
Write $\varphi_{\varepsilon}(s, t)= \varphi^s_{\varepsilon}(t)$ and observe that, by~\eqref{d1}, 
\begin{equation}\label{d4}
\av (\varphi_{\varepsilon}(\cdot, t))(1) = \int_0^1 \flgr( \cdot, \phi_0(t))^* \eta_k \leq - \int_0^1 |\nabla \eta_k|_{\M_0}^2 \leq -\delta^2 ,
\end{equation}
By \eqref{d3} and by \eqref{d4}, we conclude that
\begin{equation*}
\prim(\phi)(t) \leq \prim(\phi_0)(t) - \delta^2 \leq \su(k) + \varepsilon - \delta^2 < \su(k)  . 
\end{equation*}
Summing up the element $\varphi_{\varepsilon} \in \suk$ is such that for every $t \in [0,1]$, $\prim(\varphi_{\varepsilon})(t) < \su (k) $ in contradiction with the definition of $\su(k)$. Therefore, the element $\varphi_{\varepsilon}$ is such that either $\prim(\varphi_{\varepsilon})(t) < s^{\mathrm{u}}(k)$ or $\prim(\varphi_{\varepsilon})(t) \in [s^{\mathrm{u}}(k) ,  \su +\varepsilon)$ and $\varphi_{\varepsilon}(t) \in V{'} \subset V$. The claim holds.

\end{proof}
\end{lemma}
\begin{lemma}\label{index2} Let $k\in J$. There exists $\gamma \in \Z(\eta_k)$ such that $\mathrm{index}(\gamma)\leq 1$.
\begin{proof}
We proceed by contradiction by following the same scheme and by adopting the same setting and the same notation as in the proof of Lemma \ref{index1}. By Lemma \ref{struwe2} and by transversality theorem \cite[Section~5]{trans}, we can find an element $\phi_0 \in \suk$ such that its image does not intersect $B^H_{r_{\gamma_i}}\times \{ 0\}$ for every $i$ and it satisfies that for every $t \in [0,1]$ either $\prim(\phi_0)(t)< \su(k)$ or $\prim(\phi_0)(t) \in [\su(k),  \su(k) + \varepsilon^2 \delta)$ and $\phi_0(t) \in \bigcup_{i=1}^{n} \W_{\gamma_i}$. Define $\phi_1$ by~\eqref{deformazione} and look at $\phi_0(t) \in \W_{\gamma_1}$. Let $\prim$ a primitive of $\eta_k$ defined on $\W_{\gamma_1}$ and $\alpha(\phi_0(t)):[0,\varepsilon] \to \W_{\gamma_1}$ given by 
\begin{equation*}
\alpha(\phi_0(t))(s) = \mathrm{\Phi}_{\gamma_1}^{-1} \Big( \rho_{_{\gamma_1}}^{s\varepsilon \chi_{_{\gamma_1}} (\phi_0(t))} (\mathrm{\Phi}_{\gamma_1}(\phi_0(t)))\Big).
\end{equation*}
 By~\eqref{primug} and by point $(ii)$ of Lemma \ref{chart}, it follows that
\begin{align*}
\av (\alpha(\phi_0(t))(1) &=  \prim (\alpha(\phi_0(t)(1)) - \prim(\alpha(\phi_0(t)(0)) \\
& \leq \Big( \prim \circ \mathrm{\Phi}_{\gamma_1}^{-1} \Big) \Big( \rho_{_{\gamma_1}}^{\varepsilon \chi_{_{\gamma_1}} (\phi_0(t)))} (\mathrm{\Phi}_{\gamma_1}(\phi_0(t)) \Big) -  \Big( \prim \circ \mathrm{\Phi}_{\gamma_1}^{-1} \Big)  \Big(\Phi_{\gamma_1}( \phi_0(t)) \Big) \\
& \leq -\varepsilon^2D_{\gamma_1},
\end{align*}
which implies 
\begin{align*}
\prim (\phi_1)(t) &= \prim (\phi_0)(t) + \av (\alpha(\phi_0(t))(1) \\
& \leq \su(k) + \varepsilon^2 \delta - \varepsilon^2 D_{\gamma_1} \\
&< \su(k).
\end{align*}
Moreover, if $\phi_0(t) \in \W_{\gamma_1} \cap \W_{\gamma_i}$ for some $i\in \{2,...,n \}$, due to~\eqref{lipschitz}, we can deduce that $\phi_1(t) \in \W_{\gamma_i , 1 }$. Thus, the element $\phi_1$ enjoys that for every $t \in [0,1]$ either $\prim(\phi_1)(t) < \su(k)$ or $\prim(\phi_1)(t) \in [\su(k),  \su(k) + \varepsilon^2 \delta)$ and $\phi_1(t) \in \bigcup_{i=2}^{n} \W_{\gamma_i , 1}$. Iterating the procedure, at the step $n$ we find an element $\phi_n \in \suk$ such that for every $t \in [0,1]$, $\prim (\phi_n)(t) < \su(k)$ which again contradicts the definition of $\su$.

\end{proof}
\end{lemma}
\subsubsection{Proof of Theorem A (and Theorem A1)}
Let $k<c$ and assume $\Mric>0$. By continuity of $\mathrm{Ric}^{\Omega}_k$, there exists an open interval $I_k\subset (0,c)$ centered at $k$ and a positive constant $r$ such that $\mathrm{Ric}^{\Omega}_s \geq \frac{1}{r^2}>0$ for all $s \in I_k$. With the help of Lemma \ref{index1} or Lemma \ref{index2}, let $\{ \gamma_n = (x_n , T_n) \}$ be a sequence such that $\gamma_n$ is a contractible zero of $\eta_{k_n}$ with index smaller or equal than 1 and with $k_n \to k$. By Lemma \ref{Bonnet-Myers}, $T_n \leq 2\pi r$ so that
\begin{equation*}
|\eta_k (\gamma_n)|_{\M}  = |(\eta_k - \eta_{k_n})(\gamma_n)|_{\M} + |\eta_{k_n}(\gamma_n)|_{\M} =|k_n-k| T_n \leq |k_n - k | 2\pi r  .
\end{equation*}
Thus, $\{ \gamma_n \}$ is a vanishing sequence for $\eta_k$ and $T_n$ is bounded from above. Since the magnetic flow $\mathbf{\Phi}^{\sigma}$ as defined in \eqref{flussomagnetico} does not have rest points outside the zero section, by \cite[Proposition~1,~Section~4.1]{HZ}, $\gamma_n$ converges up to subsequence to a point $\gamma \in \M_0$ which is a contractible zero of $\eta_k$. Finally, if $\sigma$ is nowhere vanishing, by point $(ii)$ of Proposition \ref{cor1}, the magnetic Ricci curvature is positive for values close to zero and Theorem A1 follows.

\bibliographystyle{plain.bst}
\bibliography{Correction.MagneticCurvature-ClosedMagneticGeodesics.bib}

\end{document}